\documentclass[12pt,reqno]{amsart}

\usepackage[a4paper,left=29mm,right=29mm,top=29mm,bottom=24mm]{geometry}
\usepackage{amsmath,amssymb,amsthm,mathtools}
\numberwithin{equation}{section}
\usepackage{bm}
\usepackage{enumitem}
\usepackage{aliascnt}
\usepackage[hidelinks]{hyperref}
\usepackage[nameinlink,capitalize,noabbrev]{cleveref}
\hypersetup{
 pdftitle={Spectral compactness and critical Lorentz defects for Maxwell--Ohm evolution with prescribed velocity},
 pdfauthor={Cheng Yu},
 pdfsubject={Spectral weak-to-strong compactness, sharp critical instability, and Lorentz-force defects for Maxwell--Ohm evolution with prescribed velocity},
 pdfkeywords={Maxwell--Ohm evolution, weak-to-strong stability, spectral compactness, critical Sobolev endpoint, Fourier capacity, Lorentz-force defect}
}

\newtheorem{theorem}{Theorem}[section]
\newaliascnt{proposition}{theorem}
\newtheorem{proposition}[proposition]{Proposition}
\aliascntresetthe{proposition}
\newaliascnt{lemma}{theorem}
\newtheorem{lemma}[lemma]{Lemma}
\aliascntresetthe{lemma}
\newaliascnt{corollary}{theorem}
\newtheorem{corollary}[corollary]{Corollary}
\aliascntresetthe{corollary}
\theoremstyle{definition}
\newaliascnt{definition}{theorem}
\newtheorem{definition}[definition]{Definition}
\aliascntresetthe{definition}
\newaliascnt{remark}{theorem}
\newtheorem{remark}[remark]{Remark}
\aliascntresetthe{remark}

\newcommand{\T}{\mathbb T}
\newcommand{\R}{\mathbb R}
\newcommand{\Hh}{\mathcal H}
\newcommand{\eps}{\varepsilon}
\newcommand{\weak}{\rightharpoonup}
\newcommand{\weakstar}{\stackrel{*}{\rightharpoonup}}
\newcommand{\dd}{\,\mathrm d}
\newcommand{\cB}{\mathcal B}

\newcommand{\cT}{\mathcal T}
\newcommand{\norm}[1]{\lVert #1\rVert}

\DeclareMathOperator{\curl}{curl}
\DeclareMathOperator{\diver}{div}
\DeclareMathOperator{\supp}{supp}

\title[Spectral compactness and critical Lorentz defects]{Spectral Compactness and Critical Lorentz Defects for Maxwell--Ohm Evolution with Prescribed Velocity}
\author{Cheng Yu}
\address{Department of Mathematics, University of Florida, Gainesville, FL 32611, USA}
\email{chengyu@ufl.edu}
\date{}

\begin{document}

\begin{abstract}
We establish a spectral weak-to-strong stability criterion and a sharp
Sobolev dichotomy for Maxwell--Ohm evolution with prescribed velocity on
$\T^d$, $d=2,3$.
 Uniform spatial spectral tightness, together with weak convergence
of each fixed spatial Fourier mode and strong convergence of the
electromagnetic initial states, yields strong convergence of the
electromagnetic fields, weak convergence of the currents, and distributional
convergence of the Lorentz forces. In particular, weak convergence of the
velocity coefficients in $L^2_tH^s_x$ suffices when $s>d/2$.

At the critical index $s=d/2$, this conclusion fails sharply. We construct
smooth divergence-free velocities converging strongly to zero in the critical
Sobolev space, together with electromagnetic initial states converging
strongly to zero, while the corresponding terminal fields remain of order
one and the currents stay bounded. In two dimensions the construction is
based on a moving logarithmic core, whereas in three dimensions it uses a
Fourier-capacity core built from Leray-projected lattice modes and a polarized
six-component Maxwell packet with exact magnetic divergence constraint.
In both dimensions, the Lorentz forces converge to an explicit nonzero
measure supported on a single ray with fixed unit direction.
\end{abstract}

\subjclass[2020]{Primary 35Q61; Secondary 35B35, 35D30, 35L45}
\keywords{Maxwell--Ohm evolution, weak-to-strong stability, spectral compactness, critical Sobolev endpoint, Fourier capacity, polarized Maxwell packets, Lorentz-force defect, norm inflation}
\maketitle

\section{Introduction and main results}\label{sec:introduction}

The Maxwell--Ohm evolution with a prescribed velocity coefficient couples a constant-principal-part hyperbolic system to a rough time-dependent multiplication operator. Its natural energy controls the electric and magnetic fields in $L^\infty_tL^2_x$ and the current in $L^2_{t,x}$, but gives no spatial derivative of either electromagnetic field. Consequently, weak convergence of the coefficient and weak energy convergence of the fields do not directly identify Ohm's law or the Lorentz force.

This weak-product obstruction is also relevant to finite-energy weak solutions for coupled Navier--Stokes--Maxwell systems; see \cite{Masmoudi2010,ArsenioSaintRaymond2019,ArsenioGallagher2020}. Small-data global regular theories include \cite{IbrahimKeraani2011,GermainIbrahimMasmoudi2014}. Recent developments in this framework include refined velocity
estimates and optimal large-time decay under additional
low-frequency assumptions
\cite{HouamedIbrahimSaidHouari2026}.
Large-data results under additional geometric or structural assumptions include the axisymmetric setting of \cite{ArsenioHassainiaHouamed2024} and the planar inhomogeneous setting of \cite{ArsenioHouamedSaidHouari2025}. In a different parameter regime, Ars\'{e}nio--Ibrahim--Masmoudi \cite{ArsenioIbrahimMasmoudi2015} derive viscous magnetohydrodynamics as the speed of light $c\to\infty$ under natural finite-energy bounds, using direct frequency analysis to establish weak stability of the Lorentz force. The present paper studies continuity of the Maxwell--Ohm coefficient-to-solution map when the prescribed velocity varies in a weak topology. This isolates the closure mechanism in Ohm's law and the Lorentz force without importing parabolic compactness from a momentum equation.

For this prescribed-velocity problem, we establish a sharp
dichotomy in 2D and 3D. An intrinsic spectral tightness criterion gives strong electromagnetic compactness, weak current convergence, and Lorentz-force closure. The criterion holds under weak $L^2_tH^s_x$ coefficient convergence when $s>d/2$. At $s=d/2$, distinct endpoint constructions give strongly vanishing coefficients and initial fields, yet an order-one terminal response and an explicit nonzero Lorentz-force ray defect:
\[
 \begin{array}{ccl}
 s>\dfrac d2
 &:& \text{weak-to-strong Maxwell stability},\\[2mm]
 s=\dfrac d2
 &:& \text{sharp instability with nonzero Lorentz defects}.
 \end{array}
\]

The Maxwell--Ohm system in both dimensions takes the unified form
\begin{equation}\label{eq:intro-3d-system}
 \eps\partial_tE-\curl B_d=-j,
 \qquad
 \partial_tB_d+\curl E=0,
 \qquad
 j=\sigma\bigl(E+G_d(v,B_d)\bigr).
\end{equation}
We prescribe the initial data
\[
 E(0,x)=E_0(x),
 \qquad
 B_d(0,x)=B_{d,0}(x),
 \qquad x\in\T^d,
\]
where $E_0$ and $B_{d,0}$ are square-integrable initial fields,
with $B_{2,0}=b_0$ and $B_{3,0}=B_0$.

In 3D, $\curl=\nabla\times$ is the usual curl
operator. In 2D, we use
\[
 R(z_1,z_2)=(z_2,-z_1),
 \qquad
 \curl b=R\nabla b,
 \qquad
 \curl E=\partial_1E_2-\partial_2E_1.
\]
The velocity--magnetic coupling is
\begin{equation}\label{eq:G-def}
 G_d(v,B_d)=
 \begin{cases}
  bRv, & d=2,\\
  v\times B, & d=3.
 \end{cases}
\end{equation}

Here we  denote by $E$, $j$, and $v$ the electric field, electric
current density, and prescribed fluid velocity, respectively;
these fields take values in $\R^d$.
The magnetic variable is denoted by $B_d$, with
\(
 B_2=b\in\R,
 \;\;
 B_3=B\in\R^3.
\)
Here $B$ is the magnetic induction in 3D,
whereas $b$ is its signed component perpendicular to the
plane in the 2D polarization.
The constants $\eps>0$ and $\sigma>0$ denote the electric
permittivity and electrical conductivity, respectively,
with magnetic permeability normalized to one.
Throughout, we consider the problem in $\T^d=(\R/2\pi\mathbb Z)^d$ for $d=2, 3$.

Thus the 2D model is the planar polarization
of the 3D system with fields independent of
$x_3$, electric field and velocity in the plane, and magnetic
field perpendicular to it.
The corresponding Lorentz-force densities are $bRj$ for
$d=2$ and $j\times B$ for $d=3$.

Integration is with respect to Lebesgue measure, the Fourier
characters are $e^{ik\cdot x}$, $k\in\mathbb Z^d$, and Dirac
masses have unit mass. The map $R$ is the clockwise quarter-turn;
in 3D the cross product is oriented by
$e_1\times e_2=e_3$. These conventions make the coefficient,
direction, and sign of the defect measure below unambiguous.

Let $P_M$ denote the spatial Fourier projection onto $|k|\le M$. Using the standard Fourier and Sobolev conventions of \cite{BahouriCheminDanchin2011}, for $s>d/2$ Cauchy--Schwarz in Fourier space gives
\begin{equation}\label{eq:intro-sobolev-tail}
 \norm{(I-P_M)f}_{L^\infty(\T^d)}
 \le C M^{d/2-s}\norm{f}_{H^s(\T^d)},
\end{equation}
so weak convergence in $L^2_tH^s_x$ implies spectral tightness. The exponent vanishes at $s=d/2$. Our negative results show that this loss is genuine in both dimensions.

Let $B_2=b$, $B_3=B$, and write $X=(E,B_d)$. The Maxwell energy space is
\begin{equation}\label{eq:intro-energy-space}
 \norm{X}_{\Hh_d}^2
 :=\eps\norm{E}_{L^2(\T^d)}^2+\norm{B_d}_{L^2(\T^d)}^2.
\end{equation}

The endpoint constructions are written with $\eps=1$. This entails no loss of generality: with $\tau=t/\sqrt\eps$, $\widetilde E=\sqrt\eps\,E$, $\widetilde\sigma=\sigma/\sqrt\eps$, and $\widetilde v=\sqrt\eps\,v$, the system on $[0,T]$ is transformed into the same system with dielectric parameter one on $[0,T/\sqrt\eps]$. We then relabel the rescaled time interval and parameters.

The following two theorems make this dichotomy precise.

\begin{theorem}[Spectral weak-to-strong Maxwell stability]\label{thm:main-stability}
Suppose $v_n$ is spectrally tight with candidate limit $v$, in the sense that
\begin{equation}\label{eq:intro-st-bound}
 \sup_n\norm{v_n}_{L^2(0,T;L^\infty(\T^d))}
 +\norm{v}_{L^2(0,T;L^\infty(\T^d))}<\infty,
\end{equation}
\begin{equation}\label{eq:intro-st-tail}
 \lim_{M\to\infty}\left[
 \sup_n\norm{(I-P_M)v_n}_{L^2_tL^\infty_x}
 +\norm{(I-P_M)v}_{L^2_tL^\infty_x}\right]=0,
\end{equation}
and, for every fixed spatial Fourier mode $k\in\mathbb Z^d$,
\begin{equation}\label{eq:intro-st-modes}
 \widehat v_n(k)\weak\widehat v(k)
 \quad\text{weakly in }L^2(0,T).
\end{equation}
Let $X_n$ and $X$ be the corresponding Maxwell--Ohm energy solutions, and assume
\[
 X_n(0)\longrightarrow X(0)
 \quad\text{strongly in }\Hh_d.
\]
Then
\begin{equation}\label{eq:intro-main-strong}
 X_n\longrightarrow X
 \quad\text{strongly in }C([0,T];\Hh_d).
\end{equation}
Moreover,
\begin{equation}\label{eq:intro-main-current}
 j_n\weak j
 \quad\text{weakly in }L^2((0,T)\times\T^d),
\end{equation}
and the Lorentz forces converge in distributions:
\begin{equation}\label{eq:intro-main-force}
 b_nRj_n\longrightarrow bRj\quad(d=2),
 \qquad
 j_n\times B_n\longrightarrow j\times B\quad(d=3).
\end{equation}
In particular, if
\begin{equation}\label{eq:intro-sobolev-hyp}
 v_n\weak v\quad\text{in }L^2(0,T;H^s(\T^d)),
 \qquad s>\frac d2,
\end{equation}
then all the preceding conclusions hold.
\end{theorem}

\begin{theorem}[Sharp instability and Lorentz defects at $s=d/2$]\label{thm:main-endpoint}
The Sobolev condition $s>d/2$ in the final assertion of \cref{thm:main-stability} is sharp within the isotropic Sobolev scale. At the critical index $s=d/2$, the weak-to-strong conclusion fails in both dimensions. More precisely, use the normalization $\eps=1$ described above, fix $T>0$, let $w=e_1$, and set $x_c(t)=x_0+tw$.
\begin{enumerate}[label=\textup{(\roman*)},leftmargin=2.4em]
\item In 2D, where $s=d/2=1$, there are smooth divergence-free coefficients $u_n$ and real energy solutions $X_n=(E_n,b_n)$ such that
\[
 u_n\longrightarrow0\quad\text{strongly in }L^2(0,T;H^1(\T^2)),
 \qquad
 X_n(0)\longrightarrow0\quad\text{in }\Hh_2,
\]
while
\[
 \liminf_{n\to\infty}\norm{X_n(T)}_{\Hh_2}>0.
\]
The currents are bounded in $L^2((0,T)\times\T^2)$ and
\begin{equation}\label{eq:intro-2d-defect}
 b_nRj_n\weakstar
 -\frac{t}{2T^2}w\,\dd t\,\delta_{x=x_c(t)}
\end{equation}
as vector-valued Radon measures on $[0,T]\times\T^2$.
\item In 3D, where $s=d/2=3/2$, there are smooth divergence-free coefficients $v_n$, genuinely depending on all three spatial variables, and real energy solutions $X_n=(E_n,B_n)$ such that
\[
 v_n\longrightarrow0\quad\text{strongly in }L^2(0,T;H^{3/2}(\T^3)),
 \qquad
 X_n(0)\longrightarrow0\quad\text{in }\Hh_3,
\]
while
\[
 \liminf_{n\to\infty}\norm{X_n(T)}_{\Hh_3}>0.
\]
Moreover $\diver B_n=0$, the currents are bounded in $L^2((0,T)\times\T^3)$, and
\begin{equation}\label{eq:intro-3d-defect}
 j_n\times B_n\weakstar
 -\frac{t}{2T^2}w\,\dd t\,\delta_{x=x_c(t)}
\end{equation}
as vector-valued Radon measures on $[0,T]\times\T^3$.
\end{enumerate}
\end{theorem}

The coefficients in \cref{thm:main-endpoint} converge strongly to zero in $L^2(0,T;H^1(\T^2))$ in 2D and in $L^2(0,T;H^{3/2}(\T^3))$ in 3D, but do not satisfy a uniform $L^2_tL^\infty_x$ bound. Thus the examples rule out continuity under critical Sobolev control alone; they do not contradict \cref{thm:main-stability} under its full intrinsic spectral hypotheses.

For each fixed spatial Fourier truncation of the velocity, the
Maxwell--Ohm evolution fits the bilinear-control framework of
Ball--Marsden--Slemrod \cite{BallMarsdenSlemrod1982};
Dirr \cite[Theorem~3.1]{Dirr2023} provides trajectory compactness
for bounded $L^2$ control amplitudes and a fixed initial state.
Section~\ref{sec:stability} gives a self-contained Fourier--Volterra
proof of \cref{thm:main-stability}, using coefficient spectral
tightness to remove the cutoff and obtain field, current, and
Lorentz-force convergence.
The constructions in Sections~\ref{sec:2d-endpoint}
and~\ref{sec:3d-endpoint} establish the sharpness of the Sobolev
threshold and produce explicit nonzero Lorentz-force ray defects.

The planar construction uses a moving divergence-free logarithmic core. On its inner disk the core equals its prescribed amplitude exactly,
while its squared $H^1$ norm is bounded by a constant times the
amplitude squared divided by the logarithmic scale.
The choice of logarithmic scale and time-dependent amplitude
makes the $L^2_tH^1_x$ cost vanish.
The amplitude carried along the ray is time-dependent, with its maximum
$A_n(0)$ diverging as $n\to\infty$; this large amplification is concentrated
near $t=0$. It is this concentrated early amplification, rather than a
uniformly large coefficient, that produces the order-one terminal response.
 In
3D the critical core is instead a coherent Fourier-capacity kernel built from
Leray-projected lattice modes. Cubic symmetry gives an order-one point value,
whereas its $H^{3/2}$ norm tends to zero. The outgoing eigenspace of the 3D
Maxwell symbol has multiplicity two, but the projected Ohmic operator is scalar
on that whole space. This allows a polarized six-component packet to propagate
while preserving $\diver B=0$.

The coefficient cores and packet radii degenerate with the sequence index, so a uniform geometric-optics estimate is unavailable. We use an ordered diagonal construction: for each fixed core, first choose an envelope where the coefficient equals its ray value in 2D or is uniformly close to it in 3D; then send the lattice carrier frequency to infinity in an exact constant-coefficient evolution; finally correct the coefficient mismatch by an energy estimate. Every loss depending on the fixed core is absorbed by the later carrier choice. The packet construction is related to polarized geometric optics and Gaussian-beam propagation \cite{JefferisJin2015,LiuPryporov2017,Ralston1982,Rauch2012}; here the constant principal symbol permits compactly supported envelopes together with exact Fourier-mode evolution. The critical cores are related to Sobolev capacity and the concentration behind the Moser--Trudinger inequality \cite{AdamsHedberg1996,Moser1971}. Maxwell-specific high-frequency energy and polarization have also been studied through semiclassical and Wigner measures \cite{Taha2005,AkianSavin2024}; those works concern high-frequency propagation and energy measures rather than the weak prescribed-velocity continuity problem considered here. The ray-supported limits are related in spirit to microlocal defect measures \cite{Gerard1991}, but here the nonlinear Lorentz defects are computed directly and explicitly.

To the best of our knowledge, the combination established here---an intrinsic spectral compactness criterion for prescribed-velocity Maxwell--Ohm evolution, sharp critical $H^{d/2}$ counterexamples in 2D and 3D, and explicit Lorentz-force ray defects---has not previously been obtained. The planar construction also has two secondary consequences: instability around a fixed nonzero decaying electric trajectory and fixed-datum norm inflation by the uniform boundedness principle. These consequences are recorded in \cref{subsec:2d-fixed-data}.
For the related one-dimensional companion manuscript, see \cite{Yu1DCompanion2026}.

The paper is organized as follows. Section~\ref{sec:stability} proves energy well-posedness and \cref{thm:main-stability}. Section~\ref{sec:2d-endpoint} treats the critical index $s=d/2=1$, constructs the planar logarithmic core, proves the constant-core packet and weighted exactification, and establishes part~\textup{(i)} of \cref{thm:main-endpoint}. The two planar fixed-data consequences follow in \cref{subsec:2d-fixed-data}. Section~\ref{sec:3d-endpoint} treats $s=d/2=3/2$, constructs the genuinely 3D $H^{3/2}$ capacity core and six-component packet, and proves part~\textup{(ii)}.

\section{Spectral weak-to-strong stability}\label{sec:stability}

We treat the planar and physical polarizations through one compactness mechanism. Define
\[
 R(z_1,z_2)=(z_2,-z_1),
 \qquad
 \curl E=\partial_1E_2-\partial_2E_1.
\]
On $\T^2$ the planar system is
\begin{equation}\label{eq:planar-system}
 \eps\partial_tE-R\nabla b=-j,
 \qquad
 \partial_tb+\curl E=0,
 \qquad
 j=\sigma(E+bRv).
\end{equation}
It is the invariant polarization of \eqref{eq:intro-3d-system} obtained by taking $E=(E_1,E_2,0)$, $B=(0,0,b)$, and $v=(v_1,v_2,0)$. The planar Lorentz force is $bRj$. The identity
\[
 \int_{\T^2}R\nabla b\cdot E\,\dd x
 =\int_{\T^2}b\,\curl E\,\dd x
\]
gives the same energy cancellation as in 3D.  Define the coefficient operator
\begin{equation}\label{eq:Bv-operator}
 \cB_v(t)X=\bigl(-\sigma\eps^{-1}G_d(v,B_d),0\bigr),
\end{equation}
where $G_d$ is given by \eqref{eq:G-def}.
Then
\begin{equation}\label{eq:Bv-bound}
 \norm{\cB_v(t)X}_{\Hh_d}
 \le C_{\eps,\sigma}\norm{v(t)}_\infty\norm{X}_{\Hh_d}.
\end{equation}
The corresponding Lorentz observable is
\begin{equation}\label{eq:L-def}
 L_d(j,B_d)=
 \begin{cases}
 bRj,&d=2,\\
 j\times B,&d=3.
 \end{cases}
\end{equation}

\begin{lemma}[Energy solutions]\label{lem:energy-solutions}
Let  $v\in L^2(0,T;L^\infty(\T^d))$, and $X_0\in\Hh_d$. The corresponding Maxwell--Ohm system has a unique mild energy solution
\[
 X\in C([0,T];\Hh_d),
 \qquad
 j\in L^2((0,T)\times\T^d).
\]
On sets on which $\norm{v}_{L^2_tL^\infty_x}$ and $\norm{X_0}_{\Hh_d}$ are bounded,
\begin{equation}\label{eq:energy-solution-bound}
 \norm{X}_{C_t\Hh_d}^2
 +\norm{E}_{L^2_{t,x}}^2
 +\norm{j}_{L^2_{t,x}}^2
 \le C.
\end{equation}
If $X_v,X_z$ have coefficients $v,z$ and initial states $X_{v,0},X_{z,0}$, then, on the same bounded sets,
\begin{equation}\label{eq:coefficient-stability}
 \norm{X_v-X_z}_{C_t\Hh_d}^2
 \le C\left(
 \norm{X_{v,0}-X_{z,0}}_{\Hh_d}^2
 +\norm{v-z}_{L^2_tL^\infty_x}^2
 \right).
\end{equation}
The constant depends only on $T,\eps,\sigma$ and the displayed bounds.
\end{lemma}

\begin{proof}
The free Maxwell generators act by
\[
 \begin{aligned}
 A_{0,2}(E,b)&=(\eps^{-1}R\nabla b,-\curl E),\\
 A_{0,3}(E,B)&=(\eps^{-1}\nabla\times B,-\nabla\times E).
 \end{aligned}
\]
These are matrix-valued differential operators; their symbols on a fixed spatial Fourier mode are ordinary finite-dimensional matrices. With the electric projection $P_E(E,B_d)=(E,0)$, set
\[
 A_{\sigma,d}=A_{0,d}-\sigma\eps^{-1}P_E.
\]
On each spatial Fourier mode the energy-conjugated symbol of $A_{0,d}$ is skew-Hermitian. Consequently, with
\[
 D(A_{0,d})=\{X\in\Hh_d:A_{0,d}X\in\Hh_d\},
\]
$A_{0,d}$ is skew-adjoint. Equivalently, in 3D its domain is the natural product of $H(\curl)$ spaces; the longitudinal spectral modes cause no difficulty. The bounded dissipative perturbation $-\sigma\eps^{-1}P_E$ makes $A_{\sigma,d}$ maximal dissipative. It therefore generates the contraction semigroup $S_d(t)=e^{tA_{\sigma,d}}$, and its constant-coefficient Fourier symbol shows that $S_d(t)$ preserves every spatial Fourier mode. We use the standard generation and bounded-perturbation results in the form recorded in \cite{Pazy1983}.

The system can be written as the Volterra equation
\begin{equation}\label{eq:volterra}
 X(t)=S_d(t)X_0+\int_0^tS_d(t-s)\cB_v(s)X(s)\,\dd s.
\end{equation}
Since $S_d$ is a contraction semigroup and, by \eqref{eq:Bv-bound},
\[
 \norm{\cB_v(t)}_{\mathcal L(\Hh_d)}
 \le C_{\eps,\sigma}\norm{v(t)}_\infty
 \in L^1(0,T),
\]
the standard Picard iteration for \eqref{eq:volterra} converges in
$C([0,T];\Hh_d)$; indeed, its $r$th iterate carries the factor
\[
 \frac{1}{r!}
 \left(
 \int_0^T\norm{\cB_v(s)}_{\mathcal L(\Hh_d)}\,\dd s
 \right)^r.
\]
Hence \eqref{eq:volterra} has a unique solution in $C_t\Hh_d$.

Fourier--Galerkin approximation justifies the energy calculation
\begin{equation}\label{eq:energy-identity-basic}
 \frac12\frac{\dd}{\dd t}\norm{X}_{\Hh_d}^2
 +\sigma\norm{E}_2^2
 =-\sigma\int_{\T^d}G_d(v,B_d)\cdot E\,\dd x.
\end{equation}
Young's inequality gives
\[
 \frac{\dd}{\dd t}\norm{X}_{\Hh_d}^2
 +\sigma\norm{E}_2^2
 \le C_{\eps,\sigma}\norm{v}_\infty^2\norm{X}_{\Hh_d}^2.
\]
Gronwall's lemma controls $X$ and $E$. Since $j=\sigma(E+G_d(v,B_d))$, it also controls $j$ and proves \eqref{eq:energy-solution-bound}.

For the difference $Y=X_v-X_z=:(e,h)$, split
\[
 \cB_vX_v-\cB_zX_z=\cB_vY+(\cB_v-\cB_z)X_z.
\]
The same energy calculation gives
\[
 \frac{\dd}{\dd t}\norm{Y}_{\Hh_d}^2
 +\sigma\norm{e}_2^2
 \le C\norm{v}_\infty^2\norm{Y}_{\Hh_d}^2
 +C\norm{v-z}_\infty^2\norm{B_{d,z}}_2^2.
\]
The already established bound for $X_z$ and Gronwall's lemma prove \eqref{eq:coefficient-stability}.
\end{proof}

\begin{definition}[Spectrally tight coefficients]\label{def:spectral-tightness}
A sequence $v_n$ with candidate limit $v$ is spectrally tight on $(0,T)\times\T^d$ if
\begin{equation}\label{eq:spectral-bound}
 \sup_n\norm{v_n}_{L^2_tL^\infty_x}
 +\norm{v}_{L^2_tL^\infty_x}<\infty,
\end{equation}
\begin{equation}\label{eq:spectral-tail}
 \lim_{M\to\infty}\left(
 \sup_n\norm{(I-P_M)v_n}_{L^2_tL^\infty_x}
 +\norm{(I-P_M)v}_{L^2_tL^\infty_x}
 \right)=0,
\end{equation}
and
\begin{equation}\label{eq:modewise-convergence}
 \widehat v_n(k)\weak\widehat v(k)
 \quad\text{in }L^2(0,T)
 \quad\text{for every }k\in\mathbb Z^d.
\end{equation}
Here $P_M$ denotes the spatial Fourier projection onto $|k|\le M$.
\end{definition}

The terms involving the candidate limit $v$ in \eqref{eq:spectral-bound}--\eqref{eq:spectral-tail} are included only for symmetry: they are automatic from the corresponding uniform assumptions on $v_n$ and the fixed-mode convergence \eqref{eq:modewise-convergence}. Indeed, for fixed $L$ and $M$, weak lower semicontinuity on the finite-dimensional space of spatial trigonometric polynomials gives
\[
 \norm{P_Lv-P_Mv}_{L^2_tL^\infty_x}
 \le\liminf_{n\to\infty}
 \norm{P_Lv_n-P_Mv_n}_{L^2_tL^\infty_x}.
\]
On the other hand,
\[
 \norm{P_Lv_n-P_Mv_n}_{L^2_tL^\infty_x}
 \le
 \norm{(I-P_L)v_n}_{L^2_tL^\infty_x}
 +\norm{(I-P_M)v_n}_{L^2_tL^\infty_x}.
\]
It follows that $P_Mv$ is Cauchy in $L^2_tL^\infty_x$. Fixed-mode convergence identifies its limit with $v$, and hence $v\in L^2_tL^\infty_x$ and
$\norm{(I-P_M)v}_{L^2_tL^\infty_x}\to0$. Thus one may equivalently omit all limit-side assumptions from \eqref{eq:spectral-bound}--\eqref{eq:spectral-tail} and define $v$ through \eqref{eq:modewise-convergence}.

\begin{proof}[Proof of \cref{thm:main-stability}]
By \cref{lem:energy-solutions}, uniformly on bounded coefficient and initial-data sets,
\begin{equation}\label{eq:main-uniform-bounds}
 \norm{X}_{C_t\Hh_d}+\norm{j}_{L^2_{t,x}}\le C,
\end{equation}
and solutions with coefficients $v,z$ satisfy
\begin{equation}\label{eq:main-coeff-bound}
 \norm{X_v-X_z}_{C_t\Hh_d}
 \le C\left(
 \norm{X_{v,0}-X_{z,0}}_{\Hh_d}
 +\norm{v-z}_{L^2_tL^\infty_x}
 \right).
\end{equation}
The constants are uniform for the truncations used below: indeed,
$P_Mv=v-(I-P_M)v$, and \eqref{eq:spectral-bound}--\eqref{eq:spectral-tail} bound $P_Mv$ and $P_Mv_n$ in $L^2_tL^\infty_x$ independently of all sufficiently large $M$ and of $n$.

Fix $M$, put $V_n=P_Mv_n$ and $V=P_Mv$, and denote by $X_n^M$ and $X^M$ the exact Maxwell--Ohm solutions with coefficients $V_n$ and $V$, respectively, and initial states $X_{0,n}$ and $X_0$. Only the coefficient is truncated: these are not finite-dimensional Galerkin solutions.

Using a real trigonometric basis, $\cB_{V_n}(t)$ is a finite sum
of scalar Fourier amplitudes multiplying fixed bounded operators
on $\Hh_d$.
We give a direct Fourier--Volterra proof below that makes the
state-frequency tails explicit.
For a fixed initial state, trajectory compactness at this cutoff
can also be obtained from Dirr \cite[Theorem~3.1]{Dirr2023}
with $p=2$.
Weak convergence of the Fourier amplitudes identifies the limit,
and uniqueness yields convergence of the whole sequence;
varying initial states are handled by
\eqref{eq:main-coeff-bound}.
For related Dyson-series compactness arguments, see
\cite{BoussaidCaponigroChambrion2019}.

For a generic input $\Xi\in\Hh_d$, put $\cB_n=\cB_{V_n}$, define $\cT_{n,0}^M(t)\Xi=S_d(t)\Xi$ and, for $r\ge1$,
\begin{align}\label{eq:duhamel-simplex}
 \cT_{n,r}^M(t)\Xi
 ={}&\int_{0<t_r<\cdots<t_1<t}
 S_d(t-t_1)\cB_n(t_1)S_d(t_1-t_2)\cB_n(t_2)\cdots\notag\\
 &\hspace{31mm}\cdots S_d(t_{r-1}-t_r)\cB_n(t_r)S_d(t_r)\Xi
 \,\dd t_r\cdots\dd t_1.
\end{align}
The standard Volterra (Dyson) series is
\begin{equation}\label{eq:time-ordered-expansion}
 U_n^M(t,0)\Xi=\sum_{r=0}^\infty\cT_{n,r}^M(t)\Xi.
\end{equation}
Here $U_n^M(t,0)$ is the evolution operator associated with the coefficient $V_n=P_Mv_n$; in particular,
\[
 X_n^M(t)=U_n^M(t,0)X_{0,n}.
\]
Since $S_d(t)$ is contractive, symmetry of the product integrand on the time cube gives
\begin{align}\label{eq:duhamel-factorial}
 \norm{\cT_{n,r}^M(t)\Xi}_{\Hh_d}
 &\le \norm{\Xi}_{\Hh_d}
 \int_{0<t_r<\cdots<t_1<t}
 \prod_{j=1}^r\norm{\cB_n(t_j)}_{\mathcal L(\Hh_d)}
 \,\dd t_r\cdots\dd t_1\notag\\
 &=\frac1{r!}
 \left(\int_0^t\norm{\cB_n(s)}_{\mathcal L(\Hh_d)}\,\dd s\right)^r
 \norm{\Xi}_{\Hh_d}
 \le\frac{A_M^r}{r!}\norm{\Xi}_{\Hh_d},
\end{align}
where
\[
 A_M=\sup_n\int_0^T
 \norm{\cB_n(t)}_{\mathcal L(\Hh_d)}\,\dd t<\infty.
\]
Thus the standard factorial factor is exactly the volume reduction from the time cube to the ordered simplex. The free semigroup preserves Fourier support, while multiplication by $V_n$ shifts it by at most $M$. Hence, if $\Xi$ is supported in $|k|\le K$, then $\cT_{n,r}^M(t)\Xi$ is supported in $|k|\le K+rM$.

Put $Q_R=I-P_R$, $B_0=\sup_n\norm{X_{0,n}}_{\Hh_d}$, and
\[
 \delta_K=\sup_n\norm{Q_KX_{0,n}}_{\Hh_d}.
\]
Strong convergence of $X_{0,n}$ implies $\delta_K\to0$. Split
\[
 X_{0,n}=P_KX_{0,n}+Q_KX_{0,n}.
\]
If $R\ge K+LM$, then $Q_R\cT_{n,r}^M(t)P_KX_{0,n}=0$ for $r\le L$ by the Fourier-support bound. On $Q_KX_{0,n}$, bound the full Volterra sum by $\sum_{r\ge0}A_M^r/r!=e^{A_M}$. Thus
\[
 \begin{aligned}
 \norm{Q_RX_n^M(t)}_{\Hh_d}
 &\le\sum_{r>L}\frac{A_M^r}{r!}\norm{P_KX_{0,n}}_{\Hh_d}
   +e^{A_M}\norm{Q_KX_{0,n}}_{\Hh_d},
 \end{aligned}
\]
and hence
\begin{equation}\label{eq:factorial-tail-estimate}
 \sup_n\sup_{t\le T}\norm{Q_RX_n^M(t)}_{\Hh_d}
 \le B_0\sum_{r>L}\frac{A_M^r}{r!}+e^{A_M}\delta_K.
\end{equation}
For fixed $M$, first choose $K$, then $L$, and finally $R\ge K+LM$. This proves
\begin{equation}\label{eq:uniform-state-tail}
 \lim_{R\to\infty}\sup_n\sup_{t\le T}
 \norm{(I-P_R)X_n^M(t)}_{\Hh_d}=0.
\end{equation}

For fixed $R$, the projected equation is finite-dimensional. From \eqref{eq:Bv-bound}, the energy bound, and Cauchy--Schwarz in time,
\begin{equation}\label{eq:finite-mode-equicontinuity}
 \norm{P_RX_n^M(t)-P_RX_n^M(s)}_{\Hh_d}
 \le C_R|t-s|+C|t-s|^{1/2}.
\end{equation}
The first term comes from integrating $P_RA_{\sigma,d}X_n^M$, since $A_{\sigma,d}$ is bounded on the finitely many retained modes. For the second term, integrate $P_R\cB_nX_n^M$ and use
\[
 \int_s^t\norm{V_n(\tau)}_\infty\,\dd\tau
 \le |t-s|^{1/2}\norm{V_n}_{L^2_tL^\infty_x}.
\]
Arzel\`a--Ascoli makes each projected family $\{P_RX_n^M\}_n$ relatively compact in $C_t\Hh_d$. Given any tolerance, \eqref{eq:uniform-state-tail} first fixes $R$ so that the complementary tails are uniformly small; a finite net for the projected family then gives a finite net for $\{X_n^M\}_n$. Thus $X_n^M$ is relatively compact in $C_t\Hh_d$.

Let a subsequence converge to $X_*^M$. To identify it, set $m_2=2$, $m_3=3$, and write
\[
 P_M(v_n-v)
 =\sum_{|k|\le M}\sum_{\alpha=1}^{m_d}
 c_{n,k,\alpha}(t)e_k(x)e_\alpha,
 \qquad
 c_{n,k,\alpha}\weak0\quad\text{in }L^2(0,T),
\]
where $e_k$ is the scalar Fourier character and $e_\alpha$ is the $\alpha$th coordinate vector. For a smooth test field $\Phi$ and each fixed $(k,\alpha)$, the strong convergence of the magnetic component $B_{d,n}^M$ gives
\[
 g_{n,k,\alpha}(t)
 :=\int_{\T^d}G_d(e_ke_\alpha,B_{d,n}^M)\cdot\Phi\,\dd x
 \longrightarrow
 g_{k,\alpha}(t)
 :=\int_{\T^d}G_d(e_ke_\alpha,B_{d,*}^M)\cdot\Phi\,\dd x
\]
strongly in $L^2(0,T)$. Consequently,
\[
 \sum_{|k|\le M}\sum_{\alpha=1}^{m_d}
 \int_0^Tc_{n,k,\alpha}(t)g_{n,k,\alpha}(t)\,\dd t
 \longrightarrow0.
\]
The preceding convergence shows that $X_*^M$ satisfies the Maxwell--Ohm equation with coefficient $P_Mv$ in the distributional sense. Moreover $X_*^M(0)=X_0$, because the convergence holds in $C([0,T];\Hh_d)$, and
$\cB_{P_Mv}X_*^M\in L^1(0,T;\Hh_d)$, because \eqref{eq:Bv-bound} gives
\[
 \norm{\cB_{P_Mv}X_*^M}_{L^1_t\Hh_d}
 \le C_{\eps,\sigma}\sqrt T\norm{P_Mv}_{L^2_tL^\infty_x}
 \norm{X_*^M}_{C_t\Hh_d}<\infty.
\]
The variation-of-constants theorem for weak $C_0$-semigroup solutions \cite{Ball1977}, equivalently testing the distributional equation against backward adjoint semigroup orbits, therefore yields
\[
 X_*^M(t)=S_d(t)X_0
 +\int_0^tS_d(t-s)\cB_{P_Mv}(s)X_*^M(s)\,\dd s.
\]
Thus $X_*^M$ is precisely the mild solution constructed in \cref{lem:energy-solutions}. Uniqueness identifies $X_*^M=X^M$ and removes the subsequence.

Finally, \eqref{eq:main-coeff-bound} and \eqref{eq:spectral-tail} give
\[
 \norm{X_n-X}_{C_t\Hh_d}
 \le
 \norm{X_n-X_n^M}_{C_t\Hh_d}
 +\norm{X_n^M-X^M}_{C_t\Hh_d}
 +\norm{X^M-X}_{C_t\Hh_d}
 \longrightarrow0
\]
by first sending $n\to\infty$ and then $M\to\infty$.

It remains to identify the nonlinear observables. First,
\[
 G_d(v_n,B_{d,n})-G_d(v,B_d)
 =G_d(v_n,B_{d,n}-B_d)+G_d(v_n-v,B_d).
\]
The first term tends strongly to zero in $L^2_{t,x}$. For the second, fix $M$ and split
$v_n-v=P_M(v_n-v)+(I-P_M)(v_n-v)$. The high-frequency part is uniformly small in $L^2_{t,x}$ by \eqref{eq:spectral-tail}. For $\Phi\in L^2_{t,x}$, the finite-band part is a finite sum over $|k|\le M$ and $1\le\alpha\le m_d$ of pairings
\[
 \int_0^Tc_{n,k,\alpha}(t)
 \left(\int_{\T^d}G_d(e_ke_\alpha,B_d)\cdot\Phi\,\dd x\right)\dd t,
\]
and the parenthesized functions belong to $L^2(0,T)$. Condition \eqref{eq:modewise-convergence} therefore makes each pairing tend to zero. Sending first $n\to\infty$ and then $M\to\infty$ proves
\[
 G_d(v_n,B_{d,n})\weak G_d(v,B_d)
 \quad\text{in }L^2_{t,x}.
\]
Together with strong field convergence, this proves \eqref{eq:intro-main-current}.

Finally, for a smooth test field $\Phi$, the difference
\[
 \int\bigl(L_d(j_n,B_{d,n})-L_d(j,B_d)\bigr)\cdot\Phi
\]
is split into the factor $B_{d,n}-B_d$, which is strong in $L^2_{t,x}$, and the factor $j_n-j$, which is weak in $L^2_{t,x}$. This proves \eqref{eq:intro-main-force}.

If \eqref{eq:intro-sobolev-hyp} holds, weak convergence gives the fixed-mode convergence and a uniform $L^2_tH^s_x$ bound. Cauchy--Schwarz in Fourier space gives
\[
 \norm{(I-P_M)f}_\infty
 \le\left(\sum_{|k|>M}\langle k\rangle^{-2s}\right)^{1/2}
 \norm{f}_{H^s}
 \le CM^{d/2-s}\norm{f}_{H^s}.
\]
Thus \eqref{eq:spectral-bound}--\eqref{eq:spectral-tail} hold and the preceding argument applies.
\end{proof}

\begin{remark}[Magnetic divergence]\label{rem:magnetic-divergence}
In 3D, $\diver B_0=0$ is propagated by Faraday's law. The semigroup and compactness arguments above do not require this constraint, so \cref{thm:main-stability} applies on the full energy space and, by restriction, on its divergence-free magnetic subspace.
\end{remark}

\begin{remark}[The critical exponent]
At $d=2$, $s=1$, and at $d=3$, $s=3/2$, the exponent in \eqref{eq:intro-sobolev-tail} is zero. The missing uniform spectral tail is precisely the capacity exploited in \cref{sec:2d-endpoint,sec:3d-endpoint}.
\end{remark}

\section{The critical endpoint \texorpdfstring{$s=d/2=1$}{s=d/2=1} in 2D}\label{sec:2d-endpoint}

We now prove the 2D part of \cref{thm:main-endpoint}, namely the sharp failure of weak-to-strong stability at $s=d/2=1$, the space $L^2_tH^1_x$. Throughout this section $\eps=1$, $w=e_1$, and
\begin{equation}\label{eq:2d-ray-phase}
 x_c(t)=x_0+tw,
 \qquad
 \phi(t,x)=w\cdot(x-x_0)-t.
\end{equation}
The choice of a coordinate direction only ensures that the carrier frequencies belong to the Fourier lattice; any rational direction may be used after rescaling. All phases use the period-$2\pi$ convention fixed in the introduction.

\subsection{A moving divergence-free logarithmic core}

\begin{lemma}[Divergence-free logarithmic core]\label{lem:moser-core}
There is a family, smooth also in the parameters $(A,L)$, of smooth divergence-free vector fields $W_{A,L}$, supported in a fixed disk, such that $W_{A,L}=Aw$ on the concentric disk of radius $ce^{-L}$,
\begin{equation}\label{eq:moser-Linfty}
 \norm{W_{A,L}}_{L^\infty(\T^2)}\le CA,
\end{equation}
and
\begin{equation}\label{eq:moser-H1}
 \norm{W_{A,L}}_{H^1(\T^2)}^2
 \le C\left(\frac{A^2}{L}+A^2e^{-2L}\right)
\end{equation}
for $A\ge1$ and $L\ge2$.
\end{lemma}

\begin{proof}
Work in a Euclidean disk whose closure is contained in one coordinate chart of $\T^2$, and fix its radius $r_0>0$. Choose $\Theta\in C^\infty(\R;[0,1])$ such that
\[
 \Theta(z)=0\quad(z\le0),
 \qquad
 \Theta(z)=z\quad\left(\frac14\le z\le\frac34\right),
 \qquad
 \Theta(z)=1\quad(z\ge1).
\]
The transition pieces may be chosen once and for all so that
$|\Theta'|+|\Theta''|\le C$ and $\Theta(z)\le Cz$ for $0\le z\le1$. For $L\ge2$ define
\begin{equation}\label{eq:chiL}
 \chi_L(r)=
 \begin{cases}
 \Theta\bigl(L^{-1}\log(r_0/r)\bigr),&0<r<r_0,\\
 0,&r\ge r_0,
 \end{cases}
\end{equation}
and put $\chi_L(0)=1$. The flatness of $\Theta$ at $0$ and $1$ makes this a smooth radial function. It equals one on $r\le r_0e^{-L}$, vanishes on $r\ge r_0$, and is smooth jointly in $(L,x)$.

On the transition annulus, with $z=L^{-1}\log(r_0/r)$, direct differentiation gives
\[
 r\chi_L'(r)=-L^{-1}\Theta'(z),
 \qquad
 r^2\chi_L''(r)=L^{-2}\Theta''(z)+L^{-1}\Theta'(z).
\]
Consequently,
\begin{equation}\label{eq:chi-derivative-bounds}
 0\le\chi_L\le1,
 \qquad
 |r\chi_L'(r)|\le\frac{C}{L},
 \qquad
 |r^2\chi_L''(r)|\le\frac{C}{L}
\end{equation}
on $r_0e^{-L}<r<r_0$.
For a radial function, the Hessian therefore satisfies
\[
 |D^2\chi_L(x)|
 \lesssim |\chi_L''(r)|+\frac{|\chi_L'(r)|}{r}
 \lesssim\frac1{Lr^2},\qquad r=|x|,
\]
on the same annulus.

Set
\[
 \psi_{A,L}(x)=-A(Rw\cdot x)\chi_L(|x|),
 \qquad
 W_{A,L}=R\nabla\psi_{A,L}.
\]
Because $R\nabla$ is the planar perpendicular gradient, $\diver W_{A,L}=0$. Since $R^2=-I$, direct differentiation gives
\begin{equation}\label{eq:WA-formula}
 W_{A,L}=Aw\chi_L-A(Rw\cdot x)R\nabla\chi_L.
\end{equation}
Thus $W_{A,L}=Aw$ in the inner disk. Moreover, \eqref{eq:chi-derivative-bounds} and $|Rw\cdot x|\le r$ imply $|W_{A,L}|\le CA$. Differentiating \eqref{eq:WA-formula} and using the radial Hessian bound gives
\[
 |\nabla W_{A,L}|
 \lesssim A|\nabla\chi_L|+A|Rw\cdot x|\,|D^2\chi_L|
 \lesssim\frac{A}{Lr}
\]
on the logarithmic annulus. Integration in polar coordinates therefore gives
\[
 \int|\nabla W_{A,L}|^2
 \lesssim A^2L^{-2}\int_{r_0e^{-L}}^{r_0}\frac{\dd r}{r}
 \lesssim\frac{A^2}{L}.
\]
For the $L^2$ part, the inner disk contributes $O(A^2e^{-2L})$. On the logarithmic annulus, \eqref{eq:WA-formula} and $r|\chi_L'|\le C/L$ yield
\[
 \int|W_{A,L}|^2
 \lesssim A^2\int_{r_0e^{-L}}^{r_0}
 \left(\chi_L(r)^2+\frac1{L^2}\right)r\,\dd r.
\]
The substitution $r=r_0e^{-Lz}$ and the bound $\Theta(z)\le Cz$ give
\[
 \int_{r_0e^{-L}}^{r_0}\chi_L(r)^2r\,\dd r
 =Lr_0^2\int_0^1\Theta(z)^2e^{-2Lz}\,\dd z
 \le CL\int_0^\infty z^2e^{-2Lz}\,\dd z
 \le\frac{C}{L^2}.
\]
Combining the last three displays proves \eqref{eq:moser-H1}, while \eqref{eq:WA-formula} and \eqref{eq:chi-derivative-bounds} prove \eqref{eq:moser-Linfty}. Formula \eqref{eq:chiL} also verifies smooth dependence on $(A,L)$.
\end{proof}

Put $\delta_n=e^{-n}$ and define
\begin{equation}\label{eq:An-Ln}
 A_n(t)=1+\frac{2}{\sigma(t+\delta_n)},
 \qquad
 L_n(t)=n+A_n(t)^2,
\end{equation}
and
\begin{equation}\label{eq:un-definition}
 u_n(t,x)=W_{A_n(t),L_n(t)}(x-x_c(t)).
\end{equation}

\begin{proposition}[Vanishing critical cost]\label{prop:2d-critical-cost}
The fields $u_n$ are smooth and divergence-free, and
\begin{equation}\label{eq:un-H1-vanish}
 u_n\longrightarrow0
 \quad\text{strongly in }L^2(0,T;H^1(\T^2)).
\end{equation}
\end{proposition}

\begin{proof}
Translation does not change the Sobolev norm. Since $L_n=n+A_n^2$, \cref{lem:moser-core} gives
\begin{equation}\label{eq:un-H1-estimate}
 \int_0^T\norm{u_n(t)}_{H^1}^2\,\dd t
 \le C\int_0^T\frac{A_n(t)^2}{n+A_n(t)^2}\,\dd t
 +C\int_0^TA_n(t)^2e^{-2(n+A_n(t)^2)}\,\dd t.
\end{equation}
The second integral is $O(e^{-2n})$, because
$\sup_{a\ge0}a^2e^{-2a^2}<\infty$. Split the first integral at $t=n^{-1/2}$. On $(0,n^{-1/2})$ its integrand is at most one, so this part is $O(n^{-1/2})$. On $(n^{-1/2},T)$ we use $t+\delta_n\ge t$ to obtain $A_n(t)^2\le C(1+t^{-2})$. Hence
\[
 \int_{n^{-1/2}}^T\frac{A_n(t)^2}{n+A_n(t)^2}\,\dd t
 \le\frac{C}{n}\int_{n^{-1/2}}^T(1+t^{-2})\,\dd t
 =O(n^{-1/2}).
\]
Inserting these estimates into \eqref{eq:un-H1-estimate} proves \eqref{eq:un-H1-vanish}.
\end{proof}

\subsection{The amplified ray}

For divergence-free $u$, the planar system \eqref{eq:planar-system} reduces to the scalar damped wave equation
\begin{equation}\label{eq:scalar-damped-wave}
 b_{tt}-\Delta b+\sigma b_t+\sigma u\cdot\nabla b=0.
\end{equation}
Indeed, differentiate Faraday's equation, take the curl of Amp\`ere's equation, and use
$\curl(bRu)=-\diver(bu)=-u\cdot\nabla b$. Conversely, suppose \eqref{eq:scalar-damped-wave} holds and Amp\`ere's equation evolves $E$. Define the Faraday residual $r:=b_t+\curl E$. Since
\[
 \curl E_t=-\Delta b-\sigma\curl E+\sigma u\cdot\nabla b,
\]
the scalar equation gives
\[
 r_t=b_{tt}-\Delta b-\sigma\curl E+\sigma u\cdot\nabla b
 =-\sigma(b_t+\curl E)=-\sigma r.
\]
Compatible initial data mean $r(0)=b_t(0)+\curl E(0)=0$, so $r\equiv0$.

Consider the phase \eqref{eq:2d-ray-phase} and the outgoing polarization $E=-Rw\,b$. To identify the first transport equation, insert the formal WKB ansatz $b=q(t)e^{i\Lambda\phi(t,x)}$ into \eqref{eq:scalar-damped-wave}. Since $\phi_t=-1$, $\nabla\phi=w$, and $|w|=1$, the $O(\Lambda^2)$ term is $\Lambda^2(|w|^2-1)q=0$. After removing the exponential factor, the remaining expression is
\[
 i\Lambda\bigl[-2q'+\sigma(u_n\cdot w-1)q\bigr]+q''+\sigma q'.
\]
Setting the $O(\Lambda)$ coefficient to zero along the moving core gives
\begin{equation}\label{eq:transport-amplitude}
 q'(t)=\frac\sigma2\bigl(u_n(t,x_c(t))\cdot w-1\bigr)q(t)
 =\frac\sigma2(A_n(t)-1)q(t).
\end{equation}
Here $u_n(t,x_c(t))=A_n(t)w$ by the exact inner-core value in \cref{lem:moser-core}. Substituting $A_n-1=2/[\sigma(t+\delta_n)]$ from \eqref{eq:An-Ln} yields $q_n'=q_n/(t+\delta_n)$, so $q_n/(t+\delta_n)$ is constant. The terminal normalization $q_n(T)=1$ therefore gives
\begin{equation}\label{eq:qn}
 q_n(t)=\frac{t+\delta_n}{T+\delta_n},
 \qquad
 q_n(0)\longrightarrow0,
 \qquad
 q_n(T)=1.
\end{equation}
The rest of the section turns this transport calculation into exact solutions. The core is fixed first and the carrier frequency is sent to infinity afterward. No constant below is required to be uniform in the core index.

\subsection{An exact packet for a spatially constant core}

For $A\in C^\infty([0,T])$, let $C_A(t)$ denote the lower-order matrix in the planar system with spatially constant coefficient $A(t)w$:
\begin{equation}\label{eq:CA-2d}
 C_A(t)(E,b)=\bigl(-\sigma E-\sigma A(t)bRw,0\bigr).
\end{equation}
With the state ordered as $(E_1,E_2,b)$ and Fourier characters $e^{i\xi\cdot x}$, the Hermitian free Maxwell symbol is
\begin{equation}\label{eq:M2-symbol}
 M(\xi)=
 \begin{pmatrix}
 0&0&-\xi_2\\
 0&0&\xi_1\\
 -\xi_2&\xi_1&0
 \end{pmatrix},
\end{equation}
so the free equation is $\partial_t\widehat Z=-iM(\xi)\widehat Z$. For $\xi\ne0$, put $\omega=\xi/|\xi|$. An orthonormal eigenbasis is
\begin{equation}\label{eq:2d-eigenbasis}
 \begin{aligned}
 r_+(\omega)&=2^{-1/2}(-R\omega,1),&\lambda_+(\xi)&=|\xi|,\\
 r_0(\omega)&=(\omega,0),&\lambda_0(\xi)&=0,\\
 r_-(\omega)&=2^{-1/2}(R\omega,1),&\lambda_-(\xi)&=-|\xi|.
 \end{aligned}
\end{equation}
The three branches are separated by gaps of size $|\xi|$ or $2|\xi|$. A direct calculation gives
\begin{equation}\label{eq:2d-diagonal-coefficient}
 \langle r_+(\omega),C_A(t)r_+(\omega)\rangle
 =\frac\sigma2\bigl(A(t)w\cdot\omega-1\bigr).
\end{equation}
At $\omega=w$, this is exactly the coefficient in \eqref{eq:transport-amplitude}.

\begin{lemma}[Uniform branch decoupling]\label{lem:2d-branch-decoupling}
Let $C\in W^{1,1}(0,T;\mathbb C^{3\times3})$, let $|\xi|\ge1$, and let $z$ solve
\[
 z'=(-iM(\xi)+C(t))z.
\]
For $\mathcal I=\{+,0,-\}$ define
\[
 c_{\alpha\beta}(t,\omega)
 =\langle r_\alpha(\omega),C(t)r_\beta(\omega)\rangle,
 \qquad
 \gamma_\alpha(t,\omega)=\int_0^tc_{\alpha\alpha}(s,\omega)\,\dd s.
\]
Then
\begin{equation}\label{eq:2d-branch-decoupling}
 \sup_{0\le t\le T}
 \left|z(t)-\sum_{\alpha\in\mathcal I}
 e^{-i\lambda_\alpha(\xi)t+\gamma_\alpha(t,\omega)}
 \langle r_\alpha(\omega),z(0)\rangle r_\alpha(\omega)
 \right|
 \le\frac{C_C}{|\xi|}|z(0)|,
\end{equation}
where $C_C$ depends only on $T$, $\norm{C}_{L^1}$, $\norm{C}_{L^\infty}$, and $\norm{C'}_{L^1}$, and is independent of $\omega\in S^1$ and $\xi$.
\end{lemma}

\begin{proof}
Introduce the interaction-picture amplitudes
\[
 a_\alpha(t)=e^{i\lambda_\alpha(\xi)t}
 \langle r_\alpha(\omega),z(t)\rangle,
 \qquad
 \zeta_\alpha(t)=e^{-\gamma_\alpha(t,\omega)}a_\alpha(t).
\]
Then $\zeta'=R_\xi(t)\zeta$, the diagonal entries of $R_\xi$ vanish, and, for $\alpha\ne\beta$,
\begin{equation}\label{eq:Rxi-entry}
 (R_\xi)_{\alpha\beta}(t)
 =e^{i(\lambda_\alpha-\lambda_\beta)t}
 e^{-\gamma_\alpha(t,\omega)+\gamma_\beta(t,\omega)}
 c_{\alpha\beta}(t,\omega).
\end{equation}
Define
\[
 h_{\alpha\beta}(t,\omega)
 :=e^{-\gamma_\alpha(t,\omega)+\gamma_\beta(t,\omega)}c_{\alpha\beta}(t,\omega).
\]
Differentiating gives
\[
 \partial_t h_{\alpha\beta}
 =e^{-\gamma_\alpha+\gamma_\beta}
 \bigl[(-c_{\alpha\alpha}+c_{\beta\beta})c_{\alpha\beta}
       +\partial_t c_{\alpha\beta}\bigr].
\]
The orthonormal eigenvectors in \eqref{eq:2d-eigenbasis} are independent of time, so
\[
 |c_{\alpha\beta}|\le\|C(t)\|,
 \qquad
 |\partial_t c_{\alpha\beta}|\le\|C'(t)\|,
 \qquad
 |e^{-\gamma_\alpha+\gamma_\beta}|\le e^{2\norm{C}_{L^1}}.
\]
These bounds give
\begin{equation}\label{eq:h-bound}
 \sup_{\omega\in S^1}
 \left(\norm{h_{\alpha\beta}(\cdot,\omega)}_{L^\infty}
 +\norm{\partial_th_{\alpha\beta}(\cdot,\omega)}_{L^1}\right)
 \le C_C.
\end{equation}
For $\alpha\ne\beta$, set $\mu_{\alpha\beta}:=\lambda_\alpha(\xi)-\lambda_\beta(\xi)$; then $|\mu_{\alpha\beta}|\ge|\xi|$. Integration by parts gives
\[
 \begin{aligned}
 \int_0^t e^{i\mu_{\alpha\beta}s}h_{\alpha\beta}(s,\omega)\,\dd s
 &=\frac{e^{i\mu_{\alpha\beta}t}h_{\alpha\beta}(t,\omega)
          -h_{\alpha\beta}(0,\omega)}{i\mu_{\alpha\beta}}\\
 &\quad-\frac1{i\mu_{\alpha\beta}}
 \int_0^t e^{i\mu_{\alpha\beta}s}\partial_s h_{\alpha\beta}(s,\omega)\,\dd s.
 \end{aligned}
\]
Applying \eqref{eq:h-bound} to each off-diagonal entry, with the diagonal entries zero, yields
\begin{equation}\label{eq:Kxi-bound}
 \sup_{0\le t\le T}
 \left\|K_\xi(t):=\int_0^tR_\xi(s)\,\dd s\right\|
 \le\frac{C_C}{|\xi|}.
\end{equation}
Moreover, Gronwall's inequality and \eqref{eq:Rxi-entry} give
$\sup_t|\zeta(t)|\le C_C|\zeta(0)|$. Since $K_\xi'=R_\xi$ and $K_\xi(0)=0$, a second integration by parts gives
\[
 \begin{aligned}
 \zeta(t)-\zeta(0)
 &=\int_0^tR_\xi(s)\zeta(s)\,\dd s
 =K_\xi(t)\zeta(t)-\int_0^tK_\xi(s)\zeta'(s)\,\dd s\\
 &=K_\xi(t)\zeta(t)-\int_0^tK_\xi(s)R_\xi(s)\zeta(s)\,\dd s.
 \end{aligned}
\]
Thus \eqref{eq:Rxi-entry}--\eqref{eq:Kxi-bound} imply
$\sup_t|\zeta(t)-\zeta(0)|\le C_C|\xi|^{-1}|\zeta(0)|$. To reconstruct the original variable, use
\[
 \begin{aligned}
 \langle r_\alpha(\omega),z(t)\rangle
 &=e^{-i\lambda_\alpha(\xi)t+\gamma_\alpha(t,\omega)}\zeta_\alpha(t),\\
 z(t)&=\sum_{\alpha\in\mathcal I}
 e^{-i\lambda_\alpha(\xi)t+\gamma_\alpha(t,\omega)}
 \zeta_\alpha(t)r_\alpha(\omega).
 \end{aligned}
\]
Since $\zeta_\alpha(0)=\langle r_\alpha(\omega),z(0)\rangle$, orthonormality gives $|\zeta(0)|=|z(0)|$. The bound $|e^{\gamma_\alpha}|\le e^{\norm{C}_{L^1}}$ shows that the difference in \eqref{eq:2d-branch-decoupling} is at most $e^{\norm{C}_{L^1}}|\zeta(t)-\zeta(0)|$, proving the claim.
\end{proof}

Fix a real $a\in C_c^\infty(B(0,1/4))$ with $\norm{a}_2=1$, and let
$a_\eta(x)=\eta^{-1}a(x/\eta)$, periodized on $\T^2$ for sufficiently small $\eta$.

\begin{lemma}[Polarized constant-core packet]\label{lem:2d-constant-packet}
Let $A\in C^\infty([0,T])$ and let
\begin{equation}\label{eq:q-general-2d}
 q'(t)=\frac\sigma2(A(t)-1)q(t).
\end{equation}
For every fixed sufficiently small $\eta>0$ there are real exact solutions
$Z_{\Lambda,\eta}^A=(E_{\Lambda,\eta}^A,b_{\Lambda,\eta}^A)$ of \eqref{eq:planar-system} with coefficient $A(t)w$, indexed by lattice frequencies $\Lambda\to\infty$, such that
\begin{equation}\label{eq:2d-packet-limit}
 \left\|Z_{\Lambda,\eta}^A
 -q(t)\operatorname{Re}\left
 \{r_+(w)a_\eta(x-x_c(t))e^{i\Lambda\phi(t,x)}\right\}
 \right\|_{C([0,T];L^2_x)}\longrightarrow0.
\end{equation}
Define the exact current by
\[
 J_{\Lambda,\eta}^A
 :=\sigma\bigl(E_{\Lambda,\eta}^A+A(t)b_{\Lambda,\eta}^ARw\bigr).
\]
These currents satisfy
\begin{equation}\label{eq:2d-packet-current-limit}
 J_{\Lambda,\eta}^A
 -\sigma(A-1)b_{\Lambda,\eta}^{\rm lead}Rw
 \longrightarrow0
 \quad\text{in }L^2_{t,x},
\end{equation}
where $b_{\Lambda,\eta}^{\rm lead}$ is the magnetic component of the leading packet in \eqref{eq:2d-packet-limit}.
\end{lemma}

\begin{proof}
Write the complex leading packet as
\begin{equation}\label{eq:2d-leading-complex-packet}
 Y_{\Lambda,\eta}^A(t,x)
 :=q(t)r_+(w)a_\eta(x-x_c(t))e^{i\Lambda\phi(t,x)}.
\end{equation}
We construct a complex exact solution $\mathcal Z_{\Lambda,\eta}^A$ and set
$Z_{\Lambda,\eta}^A=\operatorname{Re}\mathcal Z_{\Lambda,\eta}^A$.
For $\Lambda w\in\mathbb Z^2$, modulation shifts each envelope mode $\ell$ to
$\xi=\Lambda w+\ell$. The corresponding leading coefficient is
\begin{equation}\label{eq:2d-leading-Fourier}
 \widehat Y_{\Lambda,\eta}^A(t,\Lambda w+\ell)
 =q(t)\widehat a_\eta(\ell)e^{-i(\Lambda w+\ell)\cdot x_0}
 e^{-i(\Lambda+w\cdot\ell)t}r_+(w).
\end{equation}
At $t=0$ replace $r_+(w)$ by $r_+(\xi/|\xi|)$ so that each nonzero initial mode
lies exactly in the outgoing eigenspace; thus prescribe
\begin{equation}\label{eq:2d-packet-initial-Fourier}
 \widehat{\mathcal Z}_{\Lambda,\eta}^A(0,\Lambda w+\ell)
 =q(0)\widehat a_\eta(\ell)e^{-i(\Lambda w+\ell)\cdot x_0}
 r_+\left(\frac{\Lambda w+\ell}{|\Lambda w+\ell|}\right).
\end{equation}
The zero-frequency term, corresponding to $\ell=-\Lambda w$, is set to zero. Its omission tends to zero in $L^2$ because $a_\eta$ is smooth. Every other Fourier coefficient solves
\begin{equation}\label{eq:2d-mode-ode}
 \partial_t\widehat{\mathcal Z}=(-iM(\xi)+C_A(t))\widehat{\mathcal Z}.
\end{equation}
The matrix $C_A$ belongs to $W^{1,1}(0,T)$, so \cref{lem:2d-branch-decoupling} applies to every nonzero mode. The data \eqref{eq:2d-packet-initial-Fourier} have only an outgoing component. Therefore, for every fixed $\ell$, the zero and incoming components tend to zero uniformly on $[0,T]$, while the outgoing component is, up to an error tending to zero,
\begin{equation}\label{eq:2d-mode-asymptotic}
 \begin{aligned}
 &q(0)\widehat a_\eta(\ell)e^{-i(\Lambda w+\ell)\cdot x_0}
 e^{-i|\Lambda w+\ell|t}\\
 &\qquad\times\exp\left[\int_0^t\frac\sigma2
 \bigl(A(s)w\cdot\omega_{\Lambda,\ell}-1\bigr)\,\dd s\right]
 r_+(\omega_{\Lambda,\ell}),
 \end{aligned}
\end{equation}
where
$\omega_{\Lambda,\ell}=(\Lambda w+\ell)/|\Lambda w+\ell|$. We used \eqref{eq:2d-diagonal-coefficient} to identify the outgoing diagonal term.

For every fixed $\ell$,
\begin{equation}\label{eq:2d-carrier-asymptotics}
 \omega_{\Lambda,\ell}\longrightarrow w,
 \qquad
 |\Lambda w+\ell|-\Lambda-w\cdot\ell\longrightarrow0.
\end{equation}
Together with \eqref{eq:q-general-2d}, these limits show that, for each fixed $\ell$,
\begin{equation}\label{eq:2d-mode-comparison}
 \begin{aligned}
 \Delta_{\Lambda,\ell}(t)
 &:=\left|\widehat{\mathcal Z}_{\Lambda,\eta}^A(t,\Lambda w+\ell)
       -\widehat Y_{\Lambda,\eta}^A(t,\Lambda w+\ell)\right|,\\
 \sup_{0\le t\le T}\Delta_{\Lambda,\ell}(t)&\longrightarrow0.
 \end{aligned}
\end{equation}
Indeed, the phase factor inherited from the initial datum has modulus one,
and \eqref{eq:2d-carrier-asymptotics} gives uniform convergence of the remaining
oscillatory factor on $[0,T]$. Multiplying the limiting diagonal exponential
by $q(0)$ gives $q(t)$ by \eqref{eq:q-general-2d}.
The factor $e^{-iw\cdot\ell t}$ in \eqref{eq:2d-leading-Fourier}
translates the envelope to $x_c(t)$.

Since $-iM(\xi)$ is skew-Hermitian, the mode equation gives
\[
 \frac12\frac{\dd}{\dd t}|\widehat{\mathcal Z}(t,\xi)|^2
 =\operatorname{Re}\langle\widehat{\mathcal Z},C_A(t)\widehat{\mathcal Z}\rangle
 \le\norm{C_A(t)}|\widehat{\mathcal Z}(t,\xi)|^2.
\]
Gronwall's lemma yields the uniform majorant
\begin{equation}\label{eq:2d-mode-majorant}
 |\widehat{\mathcal Z}(t,\xi)|
 \le\exp\left(\int_0^T\norm{C_A(s)}\,\dd s\right)
 |\widehat{\mathcal Z}(0,\xi)|.
\end{equation}
Equations \eqref{eq:2d-packet-initial-Fourier} and \eqref{eq:2d-leading-Fourier} bound both types of Fourier coefficients by $C[A,q]|\widehat a_\eta(\ell)|$, uniformly in $\Lambda,t$.
With the Fourier-series convention fixed in the Introduction, Parseval gives
\begin{equation}\label{eq:2d-packet-Parseval}
 \begin{aligned}
 \norm{\mathcal Z_{\Lambda,\eta}^A(t)-Y_{\Lambda,\eta}^A(t)}_2^2
 &=(2\pi)^2\sum_{\ell\in\mathbb Z^2}\Delta_{\Lambda,\ell}(t)^2\\
 &\le(2\pi)^2\sum_{|\ell|\le N}\Delta_{\Lambda,\ell}(t)^2
   +C[A,q]^2\sum_{|\ell|>N}|\widehat a_\eta(\ell)|^2.
 \end{aligned}
\end{equation}
Given $\varepsilon>0$, first choose $N$ so that the second term is less than
$\varepsilon/2$, using square summability of $\widehat a_\eta$.
Then choose $\Lambda>2N$ sufficiently large that the finite sum, including its
fixed Parseval factor, is less than $\varepsilon/2$ uniformly in $t$, by
\eqref{eq:2d-mode-comparison}. The omitted zero mode corresponds to
$\ell=-\Lambda w$ and belongs to the same tail; the exact coefficient there
stays zero under \eqref{eq:2d-mode-ode}.
Thus $\mathcal Z_{\Lambda,\eta}^A-Y_{\Lambda,\eta}^A\to0$ in
$C([0,T];L^2_x)$. Taking real parts proves \eqref{eq:2d-packet-limit}.

Write $\operatorname{Re}Y_{\Lambda,\eta}^A
=(E_{\Lambda,\eta}^{\rm lead},b_{\Lambda,\eta}^{\rm lead})$.
Its electric component satisfies
$E_{\Lambda,\eta}^{\rm lead}=-Rw\,b_{\Lambda,\eta}^{\rm lead}$; hence define
$J_{\Lambda,\eta}^{\rm lead}:=\sigma(A-1)b_{\Lambda,\eta}^{\rm lead}Rw$. Since $A\in L^2(0,T)$ is fixed in this lemma, the $C_tL^2_x$ convergence of the fields gives
\[
 \norm{J_{\Lambda,\eta}^A-J_{\Lambda,\eta}^{\rm lead}}_{L^2_{t,x}}
 \le\sigma\sqrt T\norm{E_{\Lambda,\eta}^A-E_{\Lambda,\eta}^{\rm lead}}_{C_tL^2_x}
 +\sigma\norm{A}_{L^2_t}\norm{b_{\Lambda,\eta}^A-b_{\Lambda,\eta}^{\rm lead}}_{C_tL^2_x},
\]
which tends to zero and proves \eqref{eq:2d-packet-current-limit}.
\end{proof}

\subsection{Diagonal localization and correction}

Let
\begin{equation}\label{eq:rho-n}
 \rho_n=c\exp\left(-\max_{0\le t\le T}L_n(t)\right)=ce^{-L_n(0)}.
\end{equation}
Thus $u_n(t)=A_n(t)w$ on $B(x_c(t),\rho_n)$ for every $t$. Choose $0<\eta_n<\rho_n/4$ with $\eta_n\to0$, and apply \cref{lem:2d-constant-packet} with $A=A_n$, $q=q_n$, and $\eta=\eta_n$. Define the inherited packets and currents by
\begin{equation}\label{eq:2d-specialized-packets}
 \begin{aligned}
 Z_{n,\Lambda}&:=Z_{\Lambda,\eta_n}^{A_n}
                 =(E_{n,\Lambda}^Z,b_{n,\Lambda}^Z),\\
 Y_{n,\Lambda}&:=\operatorname{Re}Y_{\Lambda,\eta_n}^{A_n}
   =q_n(t)r_+(w)a_{\eta_n}(x-x_c(t))\cos(\Lambda\phi(t,x))\\
   &\phantom{:}=(E_{n,\Lambda}^{\rm lead},b_{n,\Lambda}^{\rm lead}),\\
 J_{n,\Lambda}^Z&:=J_{\Lambda,\eta_n}^{A_n}
   =\sigma(E_{n,\Lambda}^Z+A_nb_{n,\Lambda}^ZRw),\\
 J_{n,\Lambda}^{\rm lead}&:=\sigma(A_n-1)b_{n,\Lambda}^{\rm lead}Rw.
 \end{aligned}
\end{equation}
For every fixed $n$, the two conclusions of \cref{lem:2d-constant-packet} give
\begin{equation}\label{eq:2d-specialized-convergences}
 \begin{aligned}
 \norm{Z_{n,\Lambda}-Y_{n,\Lambda}}_{C_tL^2_x}&\longrightarrow0,\\
 \norm{J_{n,\Lambda}^Z-J_{n,\Lambda}^{\rm lead}}_{L^2_{t,x}}&\longrightarrow0
 \qquad(\Lambda\to\infty).
 \end{aligned}
\end{equation}

The order of choices is essential. First $n$ is fixed, which fixes $A_n$, the logarithmic core, and the possibly very small radius $\rho_n$. Next $\eta_n$ is fixed inside that radius. Only after these choices do we let the lattice carrier $\Lambda\to\infty$. No estimate from the packet lemma is required to be uniform in $n$ or $\eta_n$. Since
$\supp a_{\eta_n}(\cdot-x_c(t))\subset B(x_c(t),\eta_n/4)$, the leading packet is supported strictly inside the region on which $u_n=A_nw$ for every $t\in[0,T]$.

For later use set
\begin{equation}\label{eq:weighted-diagonal-ledger}
 \begin{aligned}
 U_n&=1+\norm{u_n}_{L^2_tL^\infty_x},\\
 H_n&=1+\norm{A_n}_{L^1(0,T)}+\norm{A_n}_{L^2(0,T)},\\
 G_n&=\exp\left(C\int_0^T(1+\norm{u_n(t)}_\infty)\,\dd t\right).
 \end{aligned}
\end{equation}
These quantities may diverge extremely fast. For each fixed $n$, however, they are finite. We may therefore choose a lattice frequency $\Lambda_n$ so large that
\begin{equation}\label{eq:diag-field-choice}
 \norm{Z_{n,\Lambda_n}-Y_{n,\Lambda_n}}_{C_tL^2_x}
 \le\frac1{nG_nU_nH_n},
\end{equation}
\begin{equation}\label{eq:diag-current-choice}
 \norm{J^Z_{n,\Lambda_n}-J^{\rm lead}_{n,\Lambda_n}}_{L^2_{t,x}}
 \le\frac1{nG_nU_n},
\end{equation}
and
\begin{equation}\label{eq:diag-oscillation-choice}
 \Lambda_n\eta_n\ge n.
\end{equation}
For fixed $n$, \eqref{eq:2d-specialized-convergences} places both errors below
the displayed positive thresholds for all sufficiently large lattice carriers.
Increasing the same carrier if necessary also gives \eqref{eq:diag-oscillation-choice},
and hence $\Lambda_n\eta_n\to\infty$. This uses no quantitative convergence rate. Write $Z_n=Z_{n,\Lambda_n}$, $Y_n=Y_{n,\Lambda_n}$,
$J_n^Z=J_{n,\Lambda_n}^Z$, and $J_n^{\rm lead}=J_{n,\Lambda_n}^{\rm lead}$,
and abbreviate the packet components in the same way.

Since $Y_n$ is supported inside the exact constant core,
$(u_n-A_nw)b_n^{\rm lead}=0$. Thus the coefficient mismatch acts only on
$Z_n-Y_n$, and \eqref{eq:moser-Linfty} and \eqref{eq:diag-field-choice} imply
\[
 (u_n-A_nw)b_n^Z=(u_n-A_nw)(b_n^Z-b_n^{\rm lead}),
 \qquad
 \norm{u_n-A_nw}_{L^p_tL^\infty_x}\le C\norm{A_n}_{L^p_t}
 \quad(p=1,2).
\]
The definition of $H_n$ therefore gives
\begin{equation}\label{eq:coefficient-mismatch}
 \norm{(u_n-A_nw)b_n^Z}_{L^1_tL^2_x}
 \le\frac{C}{nG_nU_n},
 \qquad
 \norm{(u_n-A_nw)b_n^Z}_{L^2_{t,x}}
 \le\frac{C}{nG_nU_n}.
\end{equation}
Indeed, the $L^p_tL^2_x$ norm is bounded by
$C\norm{A_n}_{L^p_t}\norm{Z_n-Y_n}_{C_tL^2_x}$ for $p=1,2$.

Let $X_n=(E_n,b_n)$ be the exact solution for the actual coefficient $u_n$ with initial datum $X_n(0)=Z_n(0)$. The difference $D_n=X_n-Z_n$ has zero initial datum and forcing
\begin{equation}\label{eq:2d-forcing}
 F_n=\bigl(-\sigma b_n^ZR(u_n-A_nw),0\bigr).
\end{equation}
Taking the $L^2$ energy scalar product of the difference equation and using boundedness of the zero-order coefficient gives, first for smooth approximations and then by density,
\[
 \frac{\dd}{\dd t}\norm{D_n(t)}_2
 \le C(1+\norm{u_n(t)}_\infty)\norm{D_n(t)}_2+\norm{F_n(t)}_2.
\]
Gronwall's lemma and \eqref{eq:coefficient-mismatch} give
\begin{equation}\label{eq:2d-field-correction}
 \norm{D_n}_{C_tL^2_x}
 \le G_n\norm{F_n}_{L^1_tL^2_x}
 \le\frac{C}{nU_n}.
\end{equation}
Consequently,
\begin{equation}\label{eq:2d-velocity-field-correction}
 \norm{u_n(b_n-b_n^Z)}_{L^2_{t,x}}\le\frac{C}{n}.
\end{equation}
If $j_n$ is the current of $X_n$ and $J_n^Z$ that of the constant-core solution, then
\begin{equation}\label{eq:2d-current-difference}
 j_n-J_n^Z
 =\sigma\bigl(E_n-E_n^Z+(b_n-b_n^Z)Ru_n+b_n^ZR(u_n-A_nw)\bigr).
\end{equation}
Equations \eqref{eq:coefficient-mismatch}--\eqref{eq:2d-current-difference} show that
\begin{equation}\label{eq:2d-current-correction}
 \norm{j_n-J_n^Z}_{L^2_{t,x}}\longrightarrow0.
\end{equation}
More precisely, the electric-field difference is $O(n^{-1}U_n^{-1})$ in $L^2_{t,x}$ by \eqref{eq:2d-field-correction}; the term $u_n(b_n-b_n^Z)$ is $O(n^{-1})$ by \eqref{eq:2d-velocity-field-correction}; and the coefficient-mismatch term tends to zero by \eqref{eq:coefficient-mismatch}. In this bookkeeping, $H_n$ absorbs multiplication by $A_n$, $G_n$ cancels
the Gronwall loss, and $U_n$ cancels multiplication by $u_n$. The later choice
of $\Lambda_n$ absorbs all three losses, without a packet estimate uniform in $n$.

\subsection{The Lorentz ray defect}

We use the following periodization convention. For $0<\eta<\pi$, expressions of
the form
\[
a_\eta^2(x-x_c(t))\cos^2(\Lambda\phi(t,x))
\]
are first defined as compactly supported functions on $\R^d$ and then
periodized on $\T^d$. When $\Lambda w\in\mathbb Z^d$, as in every application
below, this agrees with the pointwise product of the periodized envelope and
the periodic oscillatory factor.

\begin{lemma}[Oscillatory concentration on the ray]\label{lem:ray-concentration}
Let $w\in S^{d-1}$ and $x_0\in\T^d$, and let
$a\in C_c^\infty(B(0,1/4)\subset\R^d)$ be real with $\norm{a}_2=1$. Set
$a_\eta(x)=\eta^{-d/2}a(x/\eta)$ and put
$x_c(t)=x_0+tw$, $\phi(t,x)=w\cdot(x-x_0)-t$.
Suppose $\eta_n\to0$ and $\Lambda_n\eta_n\to\infty$. Then, as nonnegative
Radon measures on $[0,T]\times\T^d$, 
\begin{equation}\label{eq:ray-measure}
 a_{\eta_n}^2(x-x_c(t))\cos^2(\Lambda_n\phi(t,x))\,\dd x\,\dd t
 \weakstar\frac12\,\dd t\,\delta_{x=x_c(t)}.
\end{equation}
In particular, for every $t$,
\begin{equation}\label{eq:ray-spatial-mass}
 \int_{\T^d}a_{\eta_n}^2(x-x_c(t))\cos^2(\Lambda_n\phi(t,x))\,\dd x
 \longrightarrow\frac12,
\end{equation}
and the convergence is uniform in $t$.
\end{lemma}

\begin{proof}
Let $\Psi\in C([0,T]\times\T^d)$. On the support of the packet write
$x=x_c(t)+\eta_ny$. Since $x_c(t)=x_0+tw$, $|w|=1$, and
$\phi(t,x)=w\cdot(x-x_0)-t$, one has
$\phi(t,x_c(t)+\eta_ny)=\eta_nw\cdot y$. Consequently the pairing in \eqref{eq:ray-measure} equals
\[
 \int_0^T\int_{\R^d}
 a(y)^2\cos^2(\Lambda_n\eta_nw\cdot y)
 \Psi(t,x_c(t)+\eta_ny)\,\dd y\,\dd t.
\]
The test function converges uniformly on the compact support of $a$ to
$\Psi(t,x_c(t))$. The identity $\cos^2\theta=(1+\cos2\theta)/2$ and the Riemann--Lebesgue lemma show that the remaining spatial integral converges to
$\frac12\int a^2=1/2$. This proves the measure convergence. Taking a test independent of $x$ gives \eqref{eq:ray-spatial-mass}; the scaled integral itself is independent of $t$, so the convergence is uniform.
\end{proof}

Applying \cref{lem:ray-concentration} with $d=2$ to the leading packets gives
\begin{equation}\label{eq:2d-leading-current}
 J_n^{\rm lead}=\sigma(A_n-1)b_n^{\rm lead}Rw,
 \qquad
 \norm{J_n^{\rm lead}(t)}_2^2
 =\frac{1+o(1)}{(T+\delta_n)^2},
\end{equation}
uniformly in $t$. Indeed,
\[
 b_n^{\rm lead}
 =2^{-1/2}q_n(t)a_{\eta_n}(x-x_c(t))
 \cos(\Lambda_n\phi(t,x)),
\]
and $\sigma(A_n-1)q_n=2/(T+\delta_n)$. Thus the current remains uniformly nontrivial even though the coefficient cost vanishes. Since $R^2=-I$,
\begin{equation}\label{eq:2d-leading-force}
 b_n^{\rm lead}RJ_n^{\rm lead}
 =-\sigma(A_n-1)(b_n^{\rm lead})^2w.
\end{equation}
Moreover,
\begin{equation}\label{eq:2d-force-coefficient}
 \sigma(A_n-1)q_n^2
 =\frac{2(t+\delta_n)}{(T+\delta_n)^2}
 \longrightarrow\frac{2t}{T^2}.
\end{equation}
By \cref{lem:ray-concentration},
\begin{equation}\label{eq:2d-ray-defect}
 b_n^{\rm lead}RJ_n^{\rm lead}
 \weakstar-\frac{t}{2T^2}w\,\dd t\,\delta_{x=x_c(t)}.
\end{equation}

\begin{proof}[Proof of \cref{thm:main-endpoint}, part~\textup{(i)}]

The coefficient convergence is \cref{prop:2d-critical-cost}. By
\cref{lem:ray-concentration} together with
\eqref{eq:2d-packet-limit}, \eqref{eq:diag-field-choice},
\eqref{eq:diag-oscillation-choice}, and \eqref{eq:qn},
$\norm{Z_n(0)}_2\to0$ while
$\liminf_n\norm{Z_n(T)}_2>0$.
Since $X_n(0)=Z_n(0)$ and
\eqref{eq:2d-field-correction} tends to zero, the initial and terminal assertions follow. Equations \eqref{eq:2d-leading-current},
\eqref{eq:diag-current-choice}, and \eqref{eq:2d-current-correction} give the current bound.

For the first replacement, from the actual coefficient to the constant core, a smooth test vector $\Psi$ gives
\[
 \begin{aligned}
 \left|\int(b_nRj_n-b_n^ZRJ_n^Z)\cdot\Psi\right|
 &\le C_\Psi\Bigl(
 \norm{b_n-b_n^Z}_{L^\infty_tL^2_x}\norm{j_n}_{L^2_{t,x}}\\
 &\qquad+\norm{b_n^Z}_{L^\infty_tL^2_x}
 \norm{j_n-J_n^Z}_{L^2_{t,x}}\Bigr)\longrightarrow0.
 \end{aligned}
\]
The identical strong--strong estimate, now using
\eqref{eq:diag-field-choice} and \eqref{eq:diag-current-choice}, replaces
$(b_n^Z,J_n^Z)$ by $(b_n^{\rm lead},J_n^{\rm lead})$. Both replacement errors tend to zero in distributions. The explicit calculation \eqref{eq:2d-ray-defect} therefore gives the claimed ray-supported limit.

Finally, the packet bounds and \eqref{eq:2d-field-correction} give
\[
 \sup_n\norm{b_n}_{L^\infty_tL^2_x}<\infty,
 \qquad
 \sup_n\norm{j_n}_{L^2_{t,x}}<\infty.
\]
Hence
\[
 \sup_n\norm{b_nRj_n}_{L^1_{t,x}}
 \le\sqrt T\sup_n\norm{b_n}_{L^\infty_tL^2_x}
 \sup_n\norm{j_n}_{L^2_{t,x}}<\infty.
\]
The distributional convergence extends by density to weak-* convergence in the space of vector-valued Radon measures: continuous test fields are approximated uniformly by smooth ones, and the displayed $L^1$ bound controls the approximation error uniformly in $n$.
\end{proof}

\subsection{Two planar fixed-data consequences}\label{subsec:2d-fixed-data}

We record two consequences of the preceding 2D construction. The second does not retain the uniformly bounded field/current package of \cref{thm:main-endpoint}.

\begin{corollary}[Instability around a fixed nonzero electric trajectory]\label{cor:fixed-background}
Let $E_*\in\R^2\setminus\{0\}$ be constant and put
\begin{equation}\label{eq:fixed-background}
 \overline X(t)=(e^{-\sigma t}E_*,0).
\end{equation}
For the same coefficients $u_n$ as in part~\textup{(i)} of \cref{thm:main-endpoint}, there are exact solutions $\widehat X_n$ such that
\begin{equation}\label{eq:fixed-background-instability}
 \widehat X_n(0)\longrightarrow(E_*,0)\quad\text{in }\Hh_2,
 \qquad
 \liminf_{n\to\infty}
 \norm{\widehat X_n(T)-\overline X(T)}_{\Hh_2}>0.
\end{equation}
Their currents are bounded in $L^2_{t,x}$, and their Lorentz forces converge to the same defect \eqref{eq:2d-ray-defect}.
\end{corollary}

\begin{proof}
Because the magnetic component of $\overline X$ vanishes and its electric component is spatially constant, \eqref{eq:fixed-background} solves the planar Maxwell--Ohm system for every coefficient $u_n$. By linearity,
$\widehat X_n=\overline X+X_n$ is therefore an exact solution, and
\eqref{eq:fixed-background-instability} follows from part~\textup{(i)} of \cref{thm:main-endpoint}. Its current is
$\widehat j_n=j_n+\sigma e^{-\sigma t}E_*$, whereas its magnetic field is $b_n$. Hence
\begin{equation}\label{eq:fixed-background-force}
 b_nR\widehat j_n=b_nRj_n+\sigma e^{-\sigma t}b_nRE_*.
\end{equation}
The leading magnetic packet has spatial $L^1$ norm bounded by
$C\eta_nq_n(t)$, while \eqref{eq:diag-field-choice} and
\eqref{eq:2d-field-correction} imply
$\norm{b_n-b_n^{\rm lead}}_{L^1_{t,x}}\to0$. Since $\eta_n\to0$ and
$0\le q_n\le1$, it follows that $b_n\to0$ in $L^1_{t,x}$. The second term in \eqref{eq:fixed-background-force} therefore tends to zero in $L^1$, while the first has the limit \eqref{eq:2d-ray-defect}. The asserted current bound is immediate.
\end{proof}

Let $\mathcal U_n(T)\in\mathcal L(\Hh_2)$ be the terminal evolution operator for the coefficient $u_n$, and let $\mathcal U_0(T)$ be the operator for the zero coefficient.

\begin{proposition}[Fixed-datum norm inflation]\label{prop:fixed-datum}
One has
\begin{equation}\label{eq:operator-norm-inflation}
 \norm{\mathcal U_n(T)-\mathcal U_0(T)}_{\mathcal L(\Hh_2)}
 \longrightarrow\infty.
\end{equation}
Consequently, there exists one fixed nonzero initial state $X_*\in\Hh_2$ such that
\begin{equation}\label{eq:fixed-datum-inflation}
 \sup_n\norm{\mathcal U_n(T)X_*-\mathcal U_0(T)X_*}_{\Hh_2}=\infty.
\end{equation}
In particular, for this $X_*$ the map
\[
u\longmapsto\mathcal U_u(T)X_*
\]
from smooth coefficients equipped with the $L^2_tH^1_x$ topology into
$\Hh_2$ is discontinuous at $u=0$.
\end{proposition}

\begin{proof}
Set $h_n=X_n(0)$ for the solutions in part~\textup{(i)} of
\cref{thm:main-endpoint}. Then $\norm{h_n}_{\Hh_2}\to0$, whereas, after discarding finitely many indices,
$\norm{\mathcal U_n(T)h_n}_{\Hh_2}\ge c>0$. The zero-coefficient evolution is bounded on $\Hh_2$, so
$\norm{\mathcal U_0(T)h_n}_{\Hh_2}\to0$. Therefore,
\[
 \norm{(\mathcal U_n(T)-\mathcal U_0(T))h_n}_{\Hh_2}
 \ge\frac c2
\]
for all sufficiently large $n$, and hence
\[
 \norm{\mathcal U_n(T)-\mathcal U_0(T)}_{\mathcal L(\Hh_2)}
 \ge\frac{c}{2\norm{h_n}_{\Hh_2}}\longrightarrow\infty.
\]

The uniform boundedness principle applied to the family of bounded linear
operators $\{\mathcal U_n(T)-\mathcal U_0(T)\}_n$ now gives an
$X_*\in\Hh_2$ satisfying \eqref{eq:fixed-datum-inflation}.
Such an $X_*$ is necessarily nonzero.
Since $u_n\to0$ in $L^2(0,T;H^1(\T^2))$ by
\cref{prop:2d-critical-cost}, \eqref{eq:fixed-datum-inflation}
proves the final discontinuity assertion.
\end{proof}

\begin{remark}[Scope of the fixed-data consequences]
The fixed-trajectory construction is explicit and retains bounded currents and the Lorentz defect, but its initial perturbation still depends on $n$. The fixed-datum statement uses one identical energy datum, but it is nonconstructive and gives an unbounded terminal sequence rather than a uniformly energy-bounded sequence with a convergent Lorentz defect.
\end{remark}

\section{The critical endpoint \texorpdfstring{$s=d/2=3/2$}{s=d/2=3/2} in genuine 3D}\label{sec:3d-endpoint}

We now prove the 3D part of \cref{thm:main-endpoint} at the isotropic
multiplier endpoint $s=d/2=3/2$ on $\T^3$. Throughout this section
$\eps=1$, $w=e_1$, $p=e_2$, $r=e_3=w\times p$, and, as in
\eqref{eq:2d-ray-phase},
\[
x_c(t)=x_0+tw,
\qquad
\phi(t,x)=w\cdot(x-x_0)-t.
\]

Unlike an $x_3$-independent lift, the coefficient below has nonzero Fourier modes spanning all three spatial directions. We call a coefficient genuinely 3D if it is not invariant under the continuous translations $x\mapsto x+sa$ for any nonzero direction $a\in\R^3$.

\subsection{A divergence-free Fourier-capacity core}

Fix $\vartheta\in C_c^\infty([0,\infty);[0,1])$ such that
$\vartheta=1$ on $[0,1]$ and $\vartheta=0$ on $[2,\infty)$. For
$k\in\mathbb Z^3\setminus\{0\}$, set
$k^\perp:=\{a\in\R^3:a\cdot k=0\}$ and let
\begin{equation}\label{eq:Leray-k}
 \Pi_k=I-\frac{k\otimes k}{|k|^2}
\end{equation}
be the orthogonal projection onto $k^\perp$.

\begin{lemma}[3D critical core]\label{lem:3d-critical-core}
For $L\ge2$, define
\begin{equation}\label{eq:SL}
 S_L=\sum_{k\in\mathbb Z^3\setminus\{0\}}
 \vartheta(e^{-L}|k|)|k|^{-3},
\end{equation}
and
\begin{equation}\label{eq:VL}
 V_L(x)=\frac{3}{2S_L}
 \sum_{k\in\mathbb Z^3\setminus\{0\}}
 \vartheta(e^{-L}|k|)|k|^{-3}\Pi_kw\,e^{ik\cdot x}.
\end{equation}
Then $V_L$ is real, smooth, and divergence-free, depends smoothly on $L$, and
\begin{equation}\label{eq:VL-bounds}
 V_L(0)=w,
 \qquad
 \norm{V_L}_{L^\infty(\T^3)}\le C,
 \qquad
 \norm{V_L}_{H^{3/2}(\T^3)}^2\le\frac{C}{L}.
\end{equation}
Moreover, $V_L$ is not invariant under translation in any nonzero spatial direction.
\end{lemma}

\begin{proof}
The sum in \eqref{eq:VL} is finite for each $L$. Its coefficients are even and real, and $k\cdot\Pi_kw=0$, which proves reality, smoothness, and $\diver V_L=0$. On every compact interval of $L$, only finitely many lattice modes occur, so the dependence on $L$ is smooth.

For $j\ge1$, write
\[
 N_j=\#\{k\in\mathbb Z^3:2^j\le|k|<2^{j+1}\}.
\]
Comparison with the volumes of unit cubes centered at lattice points gives
$c2^{3j}\le N_j\le C2^{3j}$. Hence every full shell below the cutoff contributes an amount between two positive constants to $S_L$. There are
$\lfloor L/\log2\rfloor-O(1)$ such shells, and the remaining transition shells contribute only $O(1)$. Consequently,
\begin{equation}\label{eq:SL-asymptotic}
 cL\le S_L\le CL.
\end{equation}

The weight in \eqref{eq:SL} is radial. Sign symmetry makes the off-diagonal entries vanish in the matrix sum below, while permutation symmetry makes its three diagonal entries equal. Taking the trace gives
\begin{equation}\label{eq:cubic-symmetry}
 \sum_{k\ne0}\vartheta(e^{-L}|k|)|k|^{-3}
 \frac{k\otimes k}{|k|^2}
 =\frac{S_L}{3}I.
\end{equation}
It follows that
\[
 V_L(0)=\frac{3}{2S_L}\left(S_L-\frac{S_L}{3}\right)w=w.
\]
Since $|\Pi_kw|\le1$, absolute summation gives
$\norm{V_L}_\infty\le3/2$. Parseval's identity, \eqref{eq:SL-asymptotic}, and
$\langle k\rangle^3\le C|k|^3$ for $k\ne0$ give
\[
 \norm{V_L}_{H^{3/2}}^2
 \le\frac{C}{S_L^2}
 \sum_{k\ne0}\vartheta(e^{-L}|k|)^2|k|^{-6}\langle k\rangle^3
 \le\frac{C}{S_L^2}\sum_{0<|k|\le2e^L}|k|^{-3}
 \le\frac{C}{L}.
\]
Finally, the Fourier coefficients at $k=e_2$, $k=e_3$, and $k=e_1+e_2$ are all nonzero for $L\ge2$. If $V_L$ were invariant in a direction $a$, then $a\cdot k=0$ for each of these modes. This successively gives $a_2=a_3=a_1=0$, proving the last assertion.
\end{proof}

Retain $\delta_n,A_n,L_n$ from \eqref{eq:An-Ln} and define the genuinely 3D coefficient
\begin{equation}\label{eq:vn-3d}
 v_n(t,x)=A_n(t)V_{L_n(t)}(x-x_c(t)).
\end{equation}
It is smooth and divergence-free, and
\begin{equation}\label{eq:vn-core-value}
 v_n(t,x_c(t))=A_n(t)w,
 \qquad
 \norm{v_n(t)}_\infty\le CA_n(t).
\end{equation}

\begin{proposition}[Vanishing 3D critical cost]\label{prop:3d-critical-cost}
The coefficients in \eqref{eq:vn-3d} satisfy
\begin{equation}\label{eq:vn-H32-vanish}
 v_n\longrightarrow0
 \quad\text{strongly in }L^2(0,T;H^{3/2}(\T^3)).
\end{equation}
\end{proposition}

\begin{proof}
Translations preserve the Sobolev norm, so \cref{lem:3d-critical-core} and $L_n=n+A_n^2$ give
\[
 \int_0^T\norm{v_n(t)}_{H^{3/2}}^2\,\dd t
 \le C\int_0^T\frac{A_n(t)^2}{n+A_n(t)^2}\,\dd t.
\]
On $(0,n^{-1/2})$ the integrand is at most one. On $(n^{-1/2},T)$,
$A_n(t)^2\le C(1+t^{-2})$, and hence
\[
 \int_{n^{-1/2}}^T\frac{A_n(t)^2}{n+A_n(t)^2}\,\dd t
 \le\frac{C}{n}\int_{n^{-1/2}}^T(1+t^{-2})\,\dd t
 =O(n^{-1/2}).
\]
This proves \eqref{eq:vn-H32-vanish}.
\end{proof}

\subsection{The polarized six-component packet}

For a spatially constant coefficient $A(t)w$, order the state as
$Z=(E,B)\in\mathbb C^6$ and write
\begin{equation}\label{eq:CA-3d}
 C_A(t)(E,B)=\bigl(-\sigma E-\sigma A(t)w\times B,0\bigr).
\end{equation}
The Hermitian free Maxwell symbol is
\begin{equation}\label{eq:M3-symbol}
 M_3(\xi)(E,B)=(-\xi\times B,\xi\times E),
\end{equation}
so each Fourier mode solves
\begin{equation}\label{eq:3d-mode-ode}
 z'=(-iM_3(\xi)+C_A(t))z.
\end{equation}
For $\omega=\xi/|\xi|$, the eigenvalues of $M_3(\xi)$ are
$|\xi|,0,-|\xi|$, each with multiplicity two. The outgoing eigenspace is
\begin{equation}\label{eq:Eplus}
 \mathcal E_+(\omega)
 =\left\{2^{-1/2}(e,\omega\times e):e\perp\omega\right\}.
\end{equation}
Let $P_+(\omega)$ be the orthogonal projection onto this space.

\begin{lemma}[3D outgoing branch decoupling]\label{lem:3d-branch-decoupling}
Let $A\in C^\infty([0,T])$, $|\xi|\ge1$, and
$z(0)\in\mathcal E_+(\omega)$. The solution of \eqref{eq:3d-mode-ode} satisfies
\begin{equation}\label{eq:3d-branch-decoupling}
 \sup_{0\le t\le T}
 \left|z(t)-e^{-i|\xi|t+\gamma_+(t,\omega)}z(0)\right|
 \le\frac{K_A}{|\xi|}|z(0)|,
\end{equation}
where
\begin{equation}\label{eq:gamma-plus-3d}
 \gamma_+(t,\omega)
 =\frac\sigma2\int_0^t\bigl(A(s)w\cdot\omega-1\bigr)\,\dd s,
\end{equation}
and $K_A$ depends on $T$, $\norm{C_A}_{L^1}$,
$\norm{C_A}_{L^\infty}$, and $\norm{\partial_tC_A}_{L^1}$, but not on $\omega$ or $\xi$.
\end{lemma}

\begin{proof}
We first record the spectral projections without choosing a global frame on $S^2$. Let
\[
 P_0(\omega)(E,B)=((E\cdot\omega)\omega,(B\cdot\omega)\omega),
 \qquad
 P_T=I-P_0.
\]
Thus
\[
 \begin{gathered}
 P_T(E,B)=\bigl(E-(E\cdot\omega)\omega,
 B-(B\cdot\omega)\omega\bigr),\\
 \operatorname{Ran}P_T=\{(E,B):E\cdot\omega=B\cdot\omega=0\}.
 \end{gathered}
\]
The range of $P_0$ is the longitudinal zero branch, whereas
$\operatorname{Ran}P_T$ is the transverse propagating space.
The vector triple-product identity gives
$M_3(\omega)^2(E,B)=P_T(E,B)$. On $\operatorname{Ran}P_T$ this is
$M_3(\omega)^2=I$, so the transverse eigenvalues are $\pm1$ and the
other two projections are
\begin{equation}\label{eq:3d-spectral-projections}
 P_\pm(\omega)=\frac12\bigl(P_T(\omega)\pm M_3(\omega)\bigr).
\end{equation}
Indeed, $P_TM_3=M_3P_T=M_3$ and $M_3^2=P_T$ give
\[
 P_\pm^2=\tfrac14(2P_T\pm2M_3)=P_\pm,
 \qquad P_+P_-=\tfrac14(P_T-M_3^2)=0,
 \qquad P_++P_-=P_T.
\]
Their ranges are the outgoing and incoming spaces
\[
 \begin{gathered}
 \operatorname{Ran}P_\pm=\mathcal E_\pm(\omega)
 =\{2^{-1/2}(e,\pm\omega\times e):e\perp\omega\},\\
 \operatorname{Ran}P_T=\mathcal E_+(\omega)\oplus\mathcal E_-(\omega).
 \end{gathered}
\]
The projections are smooth matrix-valued functions of $\omega$ and have operator norm one.

If $e,f\perp\omega$, the vector triple-product identity gives
\[
 \left\langle2^{-1/2}(f,\omega\times f),
 C_A(t)2^{-1/2}(e,\omega\times e)\right\rangle
 =\frac\sigma2\bigl(A(t)w\cdot\omega-1\bigr)
 \langle f,e\rangle_{\mathbb C^3}.
\]
Consequently,
\begin{equation}\label{eq:scalar-projected-ohm}
 P_+(\omega)C_A(t)P_+(\omega)
 =\frac\sigma2\bigl(A(t)w\cdot\omega-1\bigr)P_+(\omega).
\end{equation}

It remains to control interactions among the three rank-two blocks. For
$\alpha\in\{+,0,-\}$, put
$\mathcal E_\alpha=\operatorname{Ran}P_\alpha(\omega)$ and let
$\lambda_+=|\xi|$, $\lambda_0=0$, $\lambda_-=-|\xi|$. Let
$U_\alpha(t):\mathcal E_\alpha\to\mathcal E_\alpha$ be the fundamental matrix solving
\[
 U_\alpha'=P_\alpha C_AP_\alpha U_\alpha,
 \qquad
 U_\alpha(0)=I_{\mathcal E_\alpha}.
\]
It is invertible by uniqueness for the finite-dimensional linear system;
differentiating $U_\alpha^{-1}U_\alpha=I$ gives
\[
 (U_\alpha^{-1})'=-U_\alpha^{-1}P_\alpha C_AP_\alpha,
 \qquad U_\alpha^{-1}(0)=I_{\mathcal E_\alpha}.
\]
Since $\norm{P_\alpha}=1$, one has
$\norm{P_\alpha C_AP_\alpha}\le\norm{C_A}$. Gronwall applied to the two
equations yields
\[
 \max\{\norm{U_\alpha}_{L^\infty_t},
 \norm{U_\alpha^{-1}}_{L^\infty_t}\}
 \le e^{\norm{C_A}_{L^1_t}}.
\]
Integrating the corresponding bounds for their derivatives gives
\[
 \norm{U_\alpha'}_{L^1_t}
 +\norm{(U_\alpha^{-1})'}_{L^1_t}
 \le 2\norm{C_A}_{L^1_t}e^{\norm{C_A}_{L^1_t}}.
\]
Consequently, uniformly in $\omega$ and $\alpha$,
\begin{equation}\label{eq:Ualpha-bound}
 \norm{U_\alpha}_{L^\infty_t}
 +\norm{U_\alpha^{-1}}_{L^\infty_t}
 +\norm{U_\alpha'}_{L^1_t}
 +\norm{(U_\alpha^{-1})'}_{L^1_t}
 \le K_A.
\end{equation}
Introduce the interaction-picture variables
\[
 a_\alpha=e^{i\lambda_\alpha t}P_\alpha z,
 \qquad
 \zeta_\alpha=U_\alpha^{-1}a_\alpha.
\]
Using $P_\alpha M_3(\xi)=\lambda_\alpha P_\alpha$ in
\eqref{eq:3d-mode-ode} gives
\[
 a_\alpha'=P_\alpha C_AP_\alpha a_\alpha
 +\sum_{\beta\ne\alpha}
 e^{i(\lambda_\alpha-\lambda_\beta)t}P_\alpha C_AP_\beta a_\beta.
\]
The inverse equation cancels the diagonal term on differentiating
$\zeta_\alpha$:
\[
 \begin{aligned}
 \zeta_\alpha'
 &=U_\alpha^{-1}(a_\alpha'-P_\alpha C_AP_\alpha a_\alpha)\\
 &=\sum_{\beta\ne\alpha}
 e^{i(\lambda_\alpha-\lambda_\beta)t}
 U_\alpha^{-1}P_\alpha C_AP_\beta U_\beta\zeta_\beta.
 \end{aligned}
\]
For $\zeta=(\zeta_+,\zeta_0,\zeta_-)$ in the orthogonal direct sum of the
three branches, this is the interaction-picture system
\begin{equation}\label{eq:3d-interaction-system}
 \zeta'(t)=R_\xi(t)\zeta(t).
\end{equation}
Its diagonal blocks vanish, and for $\alpha\ne\beta$,
\[
 (R_\xi)_{\alpha\beta}
 =e^{i(\lambda_\alpha-\lambda_\beta)t}H_{\alpha\beta}(t,\omega),
 \qquad
 H_{\alpha\beta}=U_\alpha^{-1}P_\alpha C_AP_\beta U_\beta.
\]
Using \eqref{eq:Ualpha-bound} and differentiating this product shows
\begin{equation}\label{eq:Halpha-beta-bound}
 \sup_{\omega\in S^2}\max_{\alpha\ne\beta}
 \left(\norm{H_{\alpha\beta}}_{L^\infty_t}
 +\norm{\partial_tH_{\alpha\beta}}_{L^1_t}\right)
 \le K_A.
\end{equation}
Indeed, the product derivative contains one of $\partial_tC_A$, $(U_\alpha^{-1})'$, or $U_\beta'$, and the last two terms are controlled by
$\norm{C_A}_{L^\infty}\norm{C_A}_{L^1}$. This also makes the uniformity in $\omega$ explicit.

For $\alpha\ne\beta$, the spectral gap
$\mu_{\alpha\beta}=\lambda_\alpha-\lambda_\beta$ satisfies
$|\mu_{\alpha\beta}|\ge|\xi|$. Entrywise integration by parts gives
\[
 \int_0^te^{i\mu_{\alpha\beta}s}H_{\alpha\beta}(s)\,\dd s
 =\frac{e^{i\mu_{\alpha\beta}t}H_{\alpha\beta}(t)-H_{\alpha\beta}(0)}{i\mu_{\alpha\beta}}
 -\frac1{i\mu_{\alpha\beta}}
 \int_0^te^{i\mu_{\alpha\beta}s}\partial_sH_{\alpha\beta}(s)\,\dd s.
\]
Consequently, if $K_\xi(t)=\int_0^tR_\xi(s)\,\dd s$, then
\[
 \sup_{t\le T}\norm{K_\xi(t)}\le\frac{K_A}{|\xi|}.
\]
Moreover, $\norm{R_\xi}_{L^1_t}\le K_A$, so Gronwall gives
$\norm{\zeta}_{L^\infty_t}\le K_A|\zeta(0)|$.
To convert the small primitive $\norm{K_\xi}_{L^\infty_t}=O(|\xi|^{-1})$
into a bound for $\zeta(t)-\zeta(0)$, use $K_\xi'=R_\xi$, $K_\xi(0)=0$,
and integrate \eqref{eq:3d-interaction-system} by parts once more:
\[
 \begin{aligned}
 \zeta(t)-\zeta(0)
 &=\int_0^tK_\xi'(s)\zeta(s)\,\dd s\\
 &=K_\xi(t)\zeta(t)-\int_0^tK_\xi(s)\zeta'(s)\,\dd s\\
 &=K_\xi(t)\zeta(t)-\int_0^tK_\xi(s)R_\xi(s)\zeta(s)\,\dd s.
 \end{aligned}
\]
The right-hand side is bounded by
$\norm{K_\xi}_{L^\infty_t}(1+\norm{R_\xi}_{L^1_t})
\norm{\zeta}_{L^\infty_t}$, and hence
\begin{equation}\label{eq:zeta-block-bound}
 \sup_{t\le T}|\zeta(t)-\zeta(0)|\le\frac{K_A}{|\xi|}|\zeta(0)|.
\end{equation}
Finally, reconstruct the original mode as
\[
 z(t)=\sum_{\alpha\in\{+,0,-\}}
 e^{-i\lambda_\alpha t}U_\alpha(t)\zeta_\alpha(t).
\]
Because $z(0)\in\mathcal E_+$, one has
$\zeta_+(0)=z(0)$ and $\zeta_0(0)=\zeta_-(0)=0$.
Equation \eqref{eq:scalar-projected-ohm} gives
$U_+(t)=e^{\gamma_+(t,\omega)}I_{\mathcal E_+}$, so
\[
 z(t)-e^{-i|\xi|t+\gamma_+(t,\omega)}z(0)
 =\sum_\alpha e^{-i\lambda_\alpha t}U_\alpha(t)
 \bigl(\zeta_\alpha(t)-\zeta_\alpha(0)\bigr).
\]
By orthogonality of the branches, \eqref{eq:Ualpha-bound} and
\eqref{eq:zeta-block-bound} bound the supremum of its norm by
\[
 \max_\alpha\norm{U_\alpha}_{L^\infty_t}
 \sup_{t\le T}|\zeta(t)-\zeta(0)|
 \le\frac{K_A}{|\xi|}|z(0)|,
\]
which proves \eqref{eq:3d-branch-decoupling}.
\end{proof}

Set
\begin{equation}\label{eq:r0-polarization}
 r_0=2^{-1/2}(p,r)\in\mathcal E_+(w).
\end{equation}
Fix a real $a\in C_c^\infty(B(0,1/4)\subset\R^3)$ with $\norm{a}_2=1$, and put
$a_\eta(x)=\eta^{-3/2}a(x/\eta)$, periodized on $\T^3$ for sufficiently small $\eta$.

\begin{lemma}[Genuine 3D constant-core packet]\label{lem:3d-constant-packet}
Let $A\in C^\infty([0,T])$ and let
\begin{equation}\label{eq:q-general-3d}
 q'(t)=\frac\sigma2(A(t)-1)q(t).
\end{equation}
For every fixed sufficiently small $\eta>0$ there are real exact Maxwell--Ohm solutions
$Z_{\Lambda,\eta}^A=(E_{\Lambda,\eta}^A,B_{\Lambda,\eta}^A)$ with coefficient $A(t)w$, indexed by lattice carriers $\Lambda w$, such that
\begin{equation}\label{eq:3d-packet-limit}
 \norm{Z_{\Lambda,\eta}^A
 -q(t)a_\eta(x-x_c(t))\cos(\Lambda\phi(t,x))r_0}_{C([0,T];L^2_x)}
 \longrightarrow0.
\end{equation}
Their magnetic fields satisfy $\diver B_{\Lambda,\eta}^A=0$. If
\begin{equation}\label{eq:beta-lead-3d}
 \beta_{\Lambda,\eta}^{\rm lead}
 =2^{-1/2}q(t)a_\eta(x-x_c(t))\cos(\Lambda\phi(t,x)),
\end{equation}
then their currents, defined by
$J_{\Lambda,\eta}^A:=\sigma(E_{\Lambda,\eta}^A+A(t)w\times B_{\Lambda,\eta}^A)$, satisfy
\begin{equation}\label{eq:3d-current-packet-limit}
 J_{\Lambda,\eta}^A
 +\sigma(A-1)\beta_{\Lambda,\eta}^{\rm lead}p
 \longrightarrow0
 \quad\text{in }L^2_{t,x}.
\end{equation}
\end{lemma}

\begin{proof}
Construct first a complex exact solution
$\mathcal Z_{\Lambda,\eta}^A=(\mathcal E_{\Lambda,\eta}^A,\mathcal B_{\Lambda,\eta}^A)$
and take its real part at the end. Recall that
$x_c(t)=x_0+tw$ and $\phi(t,x)=w\cdot(x-x_0)-t$, so $\phi(t,x_c(t))=0$.
The complex leading packet is
\begin{equation}\label{eq:complex-leading-3d}
 \mathcal Y_{\Lambda,\eta}^A(t,x)
 =q(t)a_\eta(x-x_c(t))e^{i\Lambda\phi(t,x)}r_0.
\end{equation}
Indeed, expanding the translated envelope gives
\[
 a_\eta(x-x_c(t))e^{i\Lambda\phi(t,x)}
 =\sum_{\ell\in\mathbb Z^3}\widehat a_\eta(\ell)
 e^{i(\Lambda w+\ell)\cdot x}
 e^{-i(\Lambda w+\ell)\cdot x_0}e^{-i(\Lambda+w\cdot\ell)t}.
\]
Thus, at the frequency $\xi=\Lambda w+\ell$, its Fourier coefficient is
\begin{equation}\label{eq:leading-Fourier-3d}
 \widehat{\mathcal Y}_{\Lambda,\eta}^A(t,\Lambda w+\ell)
 =q(t)\widehat a_\eta(\ell)e^{-i(\Lambda w+\ell)\cdot x_0}
 e^{-i(\Lambda+w\cdot\ell)t}r_0.
\end{equation}
For $\xi=\Lambda w+\ell\ne0$, let
$\omega_{\Lambda,\ell}=\xi/|\xi|$ and prescribe the exact solution by
\begin{equation}\label{eq:3d-packet-initial-Fourier}
 \widehat{\mathcal Z}_{\Lambda,\eta}^A(0,\Lambda w+\ell)
 =q(0)\widehat a_\eta(\ell)e^{-i(\Lambda w+\ell)\cdot x_0}
 P_+(\omega_{\Lambda,\ell})r_0.
\end{equation}
Set the zero mode to zero. The projection in \eqref{eq:3d-packet-initial-Fourier} places every nonzero datum in $\mathcal E_+(\omega_{\Lambda,\ell})$. Apply \cref{lem:3d-branch-decoupling} with $\omega=\omega_{\Lambda,\ell}$ and substitute that datum to obtain
\begin{equation}\label{eq:3d-mode-asymptotic}
 \begin{aligned}
 \widehat{\mathcal Z}_{\Lambda,\eta}^A(t,\xi)
 ={}&q(0)\widehat a_\eta(\ell)e^{-i\xi\cdot x_0}
 e^{-i|\xi|t+\gamma_+(t,\omega_{\Lambda,\ell})}
 P_+(\omega_{\Lambda,\ell})r_0\\
 &+R_{\Lambda,\ell}(t),
 \end{aligned}
\end{equation}
where $\norm{P_+(\omega_{\Lambda,\ell})}=1$ and $|r_0|=1$ give
\begin{equation}\label{eq:3d-mode-remainder}
 \sup_{t\le T}|R_{\Lambda,\ell}(t)|
 \le\frac{K_A}{|\Lambda w+\ell|}
 |q(0)|\,|\widehat a_\eta(\ell)|.
\end{equation}
For each fixed $\ell$,
\begin{equation}\label{eq:3d-carrier-asymptotics}
 \omega_{\Lambda,\ell}\to w,
 \qquad
 |\Lambda w+\ell|-\Lambda-w\cdot\ell\to0,
 \qquad
 P_+(\omega_{\Lambda,\ell})r_0\to r_0.
\end{equation}
Moreover,
\[
 q(0)e^{\gamma_+(t,\omega_{\Lambda,\ell})}
 \longrightarrow q(0)\exp\left(\frac\sigma2\int_0^t(A(s)-1)\,\dd s\right)
 =q(t)
\]
uniformly in $t$ for fixed $\ell$. If $q(0)=0$, uniqueness of \eqref{eq:q-general-3d} gives $q\equiv0$, and the assertion is trivial.

The preceding limits show, for each fixed $\ell$, convergence of \eqref{eq:3d-mode-asymptotic} to \eqref{eq:leading-Fourier-3d}, uniformly in $t\in[0,T]$.
To justify summation, write $z_\xi(t)=\widehat{\mathcal Z}_{\Lambda,\eta}^A(t,\xi)$.
The Hermiticity of $M_3(\xi)$ in \eqref{eq:3d-mode-ode} gives
\[
 \frac12\frac{\dd}{\dd t}|z_\xi|^2
 =\operatorname{Re}\langle z_\xi,C_A(t)z_\xi\rangle
 \le\norm{C_A(t)}|z_\xi|^2.
\]
Gronwall's lemma, \eqref{eq:3d-packet-initial-Fourier}, and $\norm{P_+}=1$ therefore give
\begin{equation}\label{eq:3d-mode-majorant}
 \sup_{t\le T}
 |\widehat{\mathcal Z}_{\Lambda,\eta}^A(t,\Lambda w+\ell)|
 \le e^{\norm{C_A}_{L^1_t}}|q(0)|\,|\widehat a_\eta(\ell)|,
\end{equation}
while \eqref{eq:leading-Fourier-3d} is bounded by
$\norm{q}_{L^\infty_t}|\widehat a_\eta(\ell)|$.
For fixed $\eta$, smoothness of $a_\eta$ gives rapid decay of $\widehat a_\eta$ and, in particular, square summability. Set
\[
 \Delta_{\Lambda,\ell}(t)
 :=\widehat{\mathcal Z}_{\Lambda,\eta}^A(t,\Lambda w+\ell)
   -\widehat{\mathcal Y}_{\Lambda,\eta}^A(t,\Lambda w+\ell).
\]
For fixed $N$, the finite sum over $|\ell|\le N$ tends to zero uniformly in $t$. Parseval's identity and the common majorant give
\begin{equation}\label{eq:3d-packet-parseval}
 \begin{aligned}
 \sup_{t\le T}\norm{\mathcal Z_{\Lambda,\eta}^A(t)
                 -\mathcal Y_{\Lambda,\eta}^A(t)}_2^2
 &=(2\pi)^3\sup_{t\le T}\sum_{\ell\in\mathbb Z^3}
                  |\Delta_{\Lambda,\ell}(t)|^2\\
 &\le (2\pi)^3\sum_{|\ell|\le N}\sup_{t\le T}|\Delta_{\Lambda,\ell}(t)|^2
       +C[A,q]^2\sum_{|\ell|>N}|\widehat a_\eta(\ell)|^2.
 \end{aligned}
\end{equation}
The omitted zero mode corresponds to $\ell=-\Lambda w$; for $\Lambda>2N$ it lies in the same tail and satisfies the same majorant. Given $\epsilon>0$, first choose $N$ so the tail in \eqref{eq:3d-packet-parseval} is below $\epsilon/2$, then choose $\Lambda$ so the finite sum is below $\epsilon/2$. This proves convergence in $C_tL^2_x$. Taking
$Z_{\Lambda,\eta}^A:=\operatorname{Re}\mathcal Z_{\Lambda,\eta}^A$
and the real part of \eqref{eq:complex-leading-3d} proves \eqref{eq:3d-packet-limit}.

For each nonzero mode, write
$P_+(\omega_{\Lambda,\ell})r_0
 =2^{-1/2}(e_{\Lambda,\ell},\omega_{\Lambda,\ell}\times e_{\Lambda,\ell})$.
Its magnetic component is orthogonal to $\xi=|\xi|\omega_{\Lambda,\ell}$, so the initial complex magnetic field satisfies
$\xi\cdot\widehat{\mathcal B}_{\Lambda,\eta}^A(0,\xi)=0$.
Fourier-transforming Faraday's law gives
\[
 \partial_t\widehat{\mathcal B}_{\Lambda,\eta}^A(t,\xi)
 =-i\xi\times\widehat{\mathcal E}_{\Lambda,\eta}^A(t,\xi),
 \qquad
 \partial_t(\xi\cdot\widehat{\mathcal B}_{\Lambda,\eta}^A)
 =-i\xi\cdot(\xi\times\widehat{\mathcal E}_{\Lambda,\eta}^A)=0.
\]
The zero mode is divergence-free as well. Taking real parts therefore gives
$\diver B_{\Lambda,\eta}^A=0$ exactly.

For the real leading packet
$(E_{\Lambda,\eta}^{\rm lead},B_{\Lambda,\eta}^{\rm lead})
 :=\operatorname{Re}\mathcal Y_{\Lambda,\eta}^A$,
omitting the subscripts temporarily, \eqref{eq:beta-lead-3d} gives
\[
 E^{\rm lead}=\beta^{\rm lead}p,
 \qquad
 B^{\rm lead}=\beta^{\rm lead}r,
 \qquad
 w\times r=-p.
\]
Thus
\[
 J^{\rm lead}
 =\sigma(E^{\rm lead}+Aw\times B^{\rm lead})
 =-\sigma(A-1)\beta^{\rm lead}p.
\]
Finally, by Ohm's law and \eqref{eq:3d-packet-limit},
\[
 \norm{J_{\Lambda,\eta}^A-J_{\Lambda,\eta}^{\rm lead}}_{L^2_{t,x}}
 \le\sigma\sqrt T\norm{E_{\Lambda,\eta}^A-E_{\Lambda,\eta}^{\rm lead}}_{C_tL^2_x}
 +\sigma\norm{A}_{L^2_t}\norm{B_{\Lambda,\eta}^A-B_{\Lambda,\eta}^{\rm lead}}_{C_tL^2_x}
 \longrightarrow0.
\]
This proves \eqref{eq:3d-current-packet-limit}.
\end{proof}

\subsection{Exactification and the 3D defect}

\begin{proof}[Proof of \cref{thm:main-endpoint}, part~\textup{(ii)}]
Take the coefficients $v_n$ from \eqref{eq:vn-3d}. Their critical convergence is \cref{prop:3d-critical-cost}. Put
\[
 q_n(t)=\frac{t+\delta_n}{T+\delta_n};
\]
then \eqref{eq:q-general-3d} holds with $A=A_n$ and
\begin{equation}\label{eq:3d-q-identities}
 q_n(0)\to0,
 \qquad
 q_n(T)=1,
 \qquad
 \sigma(A_n-1)q_n=\frac2{T+\delta_n}.
\end{equation}

We carry out the diagonal localization, recording the only modification from the planar argument. Define
\begin{equation}\label{eq:3d-diagonal-ledger}
 \begin{aligned}
 U_n&=1+\norm{v_n}_{L^2_tL^\infty_x},\\
 H_n&=1+\norm{A_n}_{L^1(0,T)}+\norm{A_n}_{L^2(0,T)},\\
 G_n&=\exp\left(C\int_0^T(1+\norm{v_n(t)}_\infty)\,\dd t\right).
 \end{aligned}
\end{equation}
In the planar construction, $u_n=A_nw$ exactly on a moving ball, so the leading packet sees zero coefficient mismatch. The 3D Fourier-capacity core instead gives $v_n(t,x_c(t))=A_n(t)w$ at the center, without an exact flat spatial core. We therefore measure the local mismatch on the packet's support scale $\eta/4$ by
\begin{equation}\label{eq:3d-local-mismatch}
 \mathcal F_n(t,y):=v_n(t,x_c(t)+y)-A_n(t)w.
\end{equation}
For fixed $n$, this function is continuous and $\mathcal F_n(t,0)=0$ for every $t$ by \eqref{eq:vn-core-value}. Uniform continuity on a compact neighborhood of $[0,T]\times\{0\}$ consequently gives
\begin{equation}\label{eq:kappa-n}
 \kappa_n(\eta)
 :=\sup_{0\le t\le T}\sup_{|y|\le\eta/4}
 |v_n(t,x_c(t)+y)-A_n(t)w|
 \longrightarrow0
 \quad(\eta\to0).
\end{equation}
For each fixed $n$, choose $0<\eta_n\le1/n$ sufficiently small that
\begin{equation}\label{eq:kappa-choice}
 \kappa_n(\eta_n)\le\frac1{nG_nU_n}.
\end{equation}
Thus $\eta_n\to0$, and the local mismatch is small enough to absorb the later factors $G_n$ and $U_n$. Fix this $\eta_n$ before selecting a carrier.
Apply \cref{lem:3d-constant-packet} with $A=A_n$, $q=q_n$, and $\eta=\eta_n$, and define the real fields and currents by
\begin{equation}\label{eq:3d-specialized-packets}
 \begin{aligned}
 Z_{n,\Lambda}&:=Z_{\Lambda,\eta_n}^{A_n}
                 =(E_{n,\Lambda}^Z,B_{n,\Lambda}^Z),\\
 Y_{n,\Lambda}(t,x)&:=q_n(t)a_{\eta_n}(x-x_c(t))
                    \cos(\Lambda\phi(t,x))r_0\\
                 &=:(E_{n,\Lambda}^{\rm lead},B_{n,\Lambda}^{\rm lead}),\\
 J_{n,\Lambda}^Z&:=\sigma(E_{n,\Lambda}^Z+A_nw\times B_{n,\Lambda}^Z),\\
 J_{n,\Lambda}^{\rm lead}&:=\sigma(E_{n,\Lambda}^{\rm lead}
                          +A_nw\times B_{n,\Lambda}^{\rm lead})\\
                 &=-\frac{\sigma(A_n-1)q_n}{\sqrt2}
                   a_{\eta_n}(x-x_c(t))\cos(\Lambda\phi(t,x))p.
 \end{aligned}
\end{equation}
For each fixed $n$, \eqref{eq:3d-packet-limit} and \eqref{eq:3d-current-packet-limit} give
\begin{equation}\label{eq:3d-specialized-convergence}
 \norm{Z_{n,\Lambda}-Y_{n,\Lambda}}_{C_tL^2_x}\to0,
 \qquad
 \norm{J_{n,\Lambda}^Z-J_{n,\Lambda}^{\rm lead}}_{L^2_{t,x}}\to0
 \quad(\Lambda\to\infty).
\end{equation}
Since $U_n,H_n,G_n$ are finite for fixed $n$, choose a lattice carrier $\Lambda_n$ so large that, writing
$Z_n:=Z_{n,\Lambda_n}$, $Y_n:=Y_{n,\Lambda_n}$,
$J_n^Z:=J_{n,\Lambda_n}^Z$, and $J_n^{\rm lead}:=J_{n,\Lambda_n}^{\rm lead}$,
we have
\begin{equation}\label{eq:3d-field-choice}
 \norm{Z_n-Y_n}_{C_tL^2_x}
 \le\frac1{nG_nU_nH_n},
\end{equation}
\begin{equation}\label{eq:3d-current-choice}
 \norm{J_n^Z-J_n^{\rm lead}}_{L^2_{t,x}}
 \le\frac1{nG_nU_n},
\end{equation}
and enlarge the same carrier, if needed, so that
\begin{equation}\label{eq:3d-oscillation-choice}
 \Lambda_n\eta_n\ge n.
\end{equation}
The convergence in \eqref{eq:3d-specialized-convergence} ensures that both error bounds still hold for all sufficiently large carriers. This is the sequential choice of $\eta_n$ and then $\Lambda_n$ for each $n$, and requires no packet convergence rate uniform in $n$. We use the corresponding component notation
$Z_n=(E_n^Z,B_n^Z)$ and $Y_n=(E_n^{\rm lead},B_n^{\rm lead})$.

Split $B_n^Z=B_n^{\rm lead}+(B_n^Z-B_n^{\rm lead})$. The first term is supported in $B(x_c(t),\eta_n/4)$ and has uniformly bounded spatial $L^2$ norm. On that support, the coefficient mismatch is bounded by $\kappa_n(\eta_n)$. On the other hand,
$\norm{v_n-A_nw}_\infty\le CA_n$. Therefore, for $p=1,2$,
\begin{align*}
 \norm{(v_n-A_nw)\times B_n^Z}_{L^p_tL^2_x}
 \le{}&T^{1/p}\kappa_n(\eta_n)
 \norm{B_n^{\rm lead}}_{L^\infty_tL^2_x}\\
 &+C\norm{A_n}_{L^p_t}
 \norm{B_n^Z-B_n^{\rm lead}}_{C_tL^2_x}.
\end{align*}
Since $H_n\ge1+\norm{A_n}_{L^1_t}+\norm{A_n}_{L^2_t}$, \eqref{eq:kappa-choice} and \eqref{eq:3d-field-choice} give
\begin{equation}\label{eq:3d-mismatch-bound}
 \norm{(v_n-A_nw)\times B_n^Z}_{L^p_tL^2_x}
 \le\frac{C}{nG_nU_n},
 \qquad p=1,2.
\end{equation}

Let $X_n=(E_n,B_n)$ be the exact solution for the actual coefficient $v_n$ with $X_n(0)=Z_n(0)$. The difference $D_n=X_n-Z_n$ satisfies
\begin{equation}\label{eq:3d-forced-difference}
 \begin{aligned}
 \partial_tD_n&=A_{\sigma,3}D_n+\cB_{v_n}(t)D_n+F_n,
 &D_n(0)&=0,\\
 F_n&=\bigl(-\sigma(v_n-A_nw)\times B_n^Z,0\bigr).
 \end{aligned}
\end{equation}
Here $\eps=1$, so the $\Hh_3$ norm equals the spatial $L^2$ norm of the state. Taking the Maxwell energy inner product and using dissipativity of $A_{\sigma,3}$ yields
\[
 \frac12\frac{\dd}{\dd t}\norm{D_n}_{\Hh_3}^2
 \le C\norm{v_n(t)}_\infty\norm{D_n}_{\Hh_3}^2
      +\norm{F_n(t)}_{\Hh_3}\norm{D_n}_{\Hh_3}.
\]
Dividing after regularizing the norm at zero gives
$\frac{\dd}{\dd t}\norm{D_n}_{\Hh_3}
 \le C\norm{v_n(t)}_\infty\norm{D_n}_{\Hh_3}+\norm{F_n(t)}_{\Hh_3}$
in its integrated form. Gronwall's lemma and \eqref{eq:3d-mismatch-bound} with $p=1$ give
\begin{equation}\label{eq:3d-field-correction}
 \norm{D_n}_{C_tL^2_x}
 \le G_n\norm{F_n}_{L^1_tL^2_x}
 \le\frac{C}{nU_n}.
\end{equation}
In particular,
$\norm{v_n\times(B_n-B_n^Z)}_{L^2_{t,x}}\le C/n$. Since
\[
 j_n-J_n^Z
 =\sigma\bigl(E_n-E_n^Z
 +v_n\times(B_n-B_n^Z)
 +(v_n-A_nw)\times B_n^Z\bigr),
\]
\eqref{eq:3d-mismatch-bound}--\eqref{eq:3d-field-correction} give
\begin{equation}\label{eq:3d-current-correction}
 \norm{j_n-J_n^Z}_{L^2_{t,x}}\longrightarrow0.
\end{equation}

We next make the initial, terminal, and current bounds explicit. Since $|r_0|=1$ and $q_n(0)=\delta_n/(T+\delta_n)$,
\begin{equation}\label{eq:3d-initial-leading}
 \norm{Y_n(0)}_2\le q_n(0)\longrightarrow0.
\end{equation}
At the terminal time, $q_n(T)=1$ and \cref{lem:ray-concentration}, with $d=3$, gives
\begin{equation}\label{eq:3d-terminal-leading}
 \norm{Y_n(T)}_2^2
 =\int_{\T^3}a_{\eta_n}^2(x-x_c(T))
 \cos^2(\Lambda_n\phi(T,x))\,\dd x
 \longrightarrow\frac12.
\end{equation}
Here \eqref{eq:3d-oscillation-choice} supplies
$\Lambda_n\eta_n\to\infty$. Combining
\eqref{eq:3d-initial-leading}--\eqref{eq:3d-terminal-leading} with
\eqref{eq:3d-field-choice} proves
$\norm{Z_n(0)}_2\to0$ and
$\liminf_n\norm{Z_n(T)}_2>0$.
Since $X_n(0)=Z_n(0)$, we immediately have
$\norm{X_n(0)}_2=\norm{Z_n(0)}_2\to0$.
At the terminal time, \eqref{eq:3d-field-correction} gives
\[
 \norm{X_n(T)-Z_n(T)}_2
 \le\norm{X_n-Z_n}_{C_tL^2_x}
 \le\frac{C}{nU_n}\longrightarrow0.
\]
The reverse triangle inequality therefore yields
\[
 \liminf_n\norm{X_n(T)}_2
 \ge\liminf_n\norm{Z_n(T)}_2
       -\limsup_n\norm{X_n(T)-Z_n(T)}_2>0.
\]

The packet data have divergence-free magnetic component by \cref{lem:3d-constant-packet}; because $X_n(0)=Z_n(0)$, Faraday's law propagates $\diver B_n=0$ for the corrected solution. Finally, using \eqref{eq:3d-q-identities},
\begin{equation}\label{eq:3d-leading-current-norm}
 \begin{aligned}
 \norm{J_n^{\rm lead}(t)}_2^2
 &=\frac{2}{(T+\delta_n)^2}
 \int_{\T^3}a_{\eta_n}^2(x-x_c(t))
 \cos^2(\Lambda_n\phi(t,x))\,\dd x\\
 &=\frac{1+o(1)}{(T+\delta_n)^2},
 \end{aligned}
\end{equation}
uniformly in $t$. The uniformity follows because, after the rescaling
$x=x_c(t)+\eta_ny$, the spatial integral is independent of $t$. Equations \eqref{eq:3d-current-choice} and \eqref{eq:3d-current-correction} transfer the resulting $L^2_{t,x}$ bound to $j_n$.

It remains to identify the force. Applying \cref{lem:ray-concentration} with $d=3$ gives
\begin{equation}\label{eq:3d-ray-measure}
 a_{\eta_n}^2(x-x_c(t))\cos^2(\Lambda_n\phi(t,x))\,\dd x\,\dd t
 \weakstar\frac12\,\dd t\,\delta_{x=x_c(t)}.
\end{equation}
Writing $\beta_n^{\rm lead}$ as in \eqref{eq:beta-lead-3d}, we have
\[
 J_n^{\rm lead}=-\sigma(A_n-1)\beta_n^{\rm lead}p,
 \qquad
 B_n^{\rm lead}=\beta_n^{\rm lead}r,
 \qquad
 p\times r=w.
\]
Together with
\begin{equation}\label{eq:3d-force-coefficient}
 \sigma(A_n-1)q_n^2
 =\frac{2(t+\delta_n)}{(T+\delta_n)^2}
 \longrightarrow\frac{2t}{T^2},
\end{equation}
\eqref{eq:3d-ray-measure} yields
\begin{equation}\label{eq:3d-leading-defect}
 J_n^{\rm lead}\times B_n^{\rm lead}
 \weakstar-\frac{t}{2T^2}w\,\dd t\,\delta_{x=x_c(t)}.
\end{equation}

We record both replacements. For a smooth test field $\Psi$,
\begin{align*}
 &\left|\int(J_n^Z\times B_n^Z-J_n^{\rm lead}\times B_n^{\rm lead})\cdot\Psi\right|\\
 &\quad\le\norm\Psi_\infty\left(
 \norm{J_n^Z-J_n^{\rm lead}}_{L^2_{t,x}}\norm{B_n^Z}_{L^2_{t,x}}
 +\norm{J_n^{\rm lead}}_{L^2_{t,x}}
 \norm{B_n^Z-B_n^{\rm lead}}_{L^2_{t,x}}
 \right)\longrightarrow0
\end{align*}
by \eqref{eq:3d-field-choice}--\eqref{eq:3d-current-choice}. Similarly,
\begin{align*}
 &\left|\int(j_n\times B_n-J_n^Z\times B_n^Z)\cdot\Psi\right|\\
 &\quad\le\norm\Psi_\infty\left(
 \norm{j_n-J_n^Z}_{L^2_{t,x}}\norm{B_n}_{L^2_{t,x}}
 +\norm{J_n^Z}_{L^2_{t,x}}
 \norm{B_n-B_n^Z}_{L^2_{t,x}}
 \right)\longrightarrow0
\end{align*}
by \eqref{eq:3d-field-correction} and \eqref{eq:3d-current-correction}. Thus the leading force limit is also the distributional limit of $j_n\times B_n$.

Finally, the uniform $L^\infty_tL^2_x$ field bound and $L^2_{t,x}$ current bound give
\[
 \sup_n\norm{j_n\times B_n}_{L^1_{t,x}}<\infty.
\]
Approximating continuous test fields uniformly by smooth ones upgrades the distributional convergence to weak-* convergence of vector-valued Radon measures and proves \eqref{eq:intro-3d-defect}.
\end{proof}

\section*{Acknowledgments}

The author is partially supported by NSF grant DMS-2510425 and by the Simons
Foundation grant MPS-TSM-00007824.

\section*{Statement on the use of artificial intelligence}

Generative AI tools ChatGPT (OpenAI) and Claude (Anthropic)  were used during the preparation of this manuscript for
editorial assistance, consistency checks, LaTeX preparation and improvements in exposition.
All mathematical content and references were independently verified by the
author, who takes full responsibility for the final manuscript.


\begin{thebibliography}{99}
\footnotesize
\setlength{\itemsep}{-0.5pt}

\bibitem{AdamsHedberg1996}
David~R. Adams and Lars~Inge Hedberg.
\newblock \emph{Function Spaces and Potential Theory}, volume 314 of Grundlehren der mathematischen Wissenschaften.
\newblock Springer, Berlin, 1996.

\bibitem{AkianSavin2024}
Jean-Luc Akian and \'Eric Savin.
\newblock Wigner measures of electromagnetic waves in heterogeneous bianisotropic media.
\newblock \emph{Wave Motion}, 127:103296, 2024.

\bibitem{ArsenioGallagher2020}
Diogo Ars\'{e}nio and Isabelle Gallagher.
\newblock Solutions of Navier--Stokes--Maxwell systems in large energy spaces.
\newblock \emph{Transactions of the American Mathematical Society}, 373(6):3853--3884, 2020.

\bibitem{ArsenioHassainiaHouamed2024}
Diogo Ars\'{e}nio, Zineb Hassainia, and Haroune Houamed.
\newblock Axisymmetric incompressible viscous plasmas: Global well-posedness and asymptotics.
\newblock \emph{Forum of Mathematics, Sigma}, 12:e79, 2024.

\bibitem{ArsenioHouamedSaidHouari2025}
Diogo Ars\'{e}nio, Haroune Houamed, and Belkacem Said-Houari.
\newblock Global unique solutions to the planar inhomogeneous Navier--Stokes--Maxwell equations.
\newblock \emph{Journal of Differential Equations}, 446:113661, 2025.

\bibitem{ArsenioIbrahimMasmoudi2015}
Diogo Ars\'{e}nio, Slim Ibrahim, and Nader Masmoudi.
\newblock A derivation of the magnetohydrodynamic system from Navier--Stokes--Maxwell systems.
\newblock \emph{Archive for Rational Mechanics and Analysis}, 216(3):767--812, 2015.

\bibitem{ArsenioSaintRaymond2019}
Diogo Ars\'{e}nio and Laure Saint-Raymond.
\newblock \emph{From the Vlasov--Maxwell--Boltzmann System to Incompressible Viscous Electro-Magneto-Hydrodynamics. Volume I}.
\newblock EMS Monographs in Mathematics. European Mathematical Society, Z\"urich, 2019.






\bibitem{BahouriCheminDanchin2011}
Hajer Bahouri, Jean-Yves Chemin, and Rapha\"el Danchin.
\newblock \emph{Fourier Analysis and Nonlinear Partial Differential Equations}, volume 343 of Grundlehren der mathematischen Wissenschaften.
\newblock Springer, Heidelberg, 2011.

\bibitem{Ball1977}
John~M. Ball.
\newblock Strongly continuous semigroups, weak solutions, and the variation of constants formula.
\newblock \emph{Proceedings of the American Mathematical Society}, 63(2):370--373, 1977.

\bibitem{BallMarsdenSlemrod1982}
J.~M. Ball, J.~E. Marsden, and M. Slemrod.
\newblock Controllability for distributed bilinear systems.
\newblock \emph{SIAM Journal on Control and Optimization}, 20(4):575--597, 1982.

\bibitem{BoussaidCaponigroChambrion2019}
N. Boussaid, M. Caponigro, and T. Chambrion.
\newblock On the Ball--Marsden--Slemrod obstruction for bilinear control systems.
\newblock In \emph{2019 IEEE 58th Conference on Decision and Control (CDC)}, pages 4971--4976. IEEE, 2019.



\bibitem{Dirr2023}
G. Dirr.
\newblock Compactness of fixed point maps and the
Ball--Marsden--Slemrod conjecture.
\newblock \emph{SIAM Journal on Control and Optimization},
61(2):560--585, 2023.
\newblock \url{https://doi.org/10.1137/21M1461848}.




\bibitem{Gerard1991}
Patrick G\'{e}rard.
\newblock Microlocal defect measures.
\newblock \emph{Communications in Partial Differential Equations}, 16(11):1761--1794, 1991.

\bibitem{GermainIbrahimMasmoudi2014}
Pierre Germain, Slim Ibrahim, and Nader Masmoudi.
\newblock Well-posedness of the Navier--Stokes--Maxwell equations.
\newblock \emph{Proceedings of the Royal Society of Edinburgh Section A: Mathematics}, 144(1):71--86, 2014.




\bibitem{HouamedIbrahimSaidHouari2026}
Haroune Houamed, Slim Ibrahim, and Belkacem Said-Houari.
\newblock Optimal time-decay of global solutions to the
Navier--Stokes--Maxwell system.
\newblock arXiv preprint arXiv:2609.15095, 2026.
\newblock \url{https://doi.org/10.48550/arXiv.2609.15095}.



\bibitem{IbrahimKeraani2011}
Slim Ibrahim and Sahbi Keraani.
\newblock Global small solutions for the Navier--Stokes--Maxwell system.
\newblock \emph{SIAM Journal on Mathematical Analysis}, 43(5):2275--2295, 2011.

\bibitem{JefferisJin2015}
Leland Jefferis and Shi Jin.
\newblock A Gaussian beam method for high frequency solution of symmetric hyperbolic systems with polarized waves.
\newblock \emph{Multiscale Modeling \& Simulation}, 13(3):733--765, 2015.

\bibitem{LiuPryporov2017}
Hailiang Liu and Maksym Pryporov.
\newblock Error estimates for Gaussian beam methods applied to symmetric strictly hyperbolic systems.
\newblock \emph{Wave Motion}, 73:57--75, 2017.

\bibitem{Masmoudi2010}
Nader Masmoudi.
\newblock Global well posedness for the Maxwell--Navier--Stokes system in 2D.
\newblock \emph{Journal de Math\'{e}matiques Pures et Appliqu\'{e}es}, 93(6):559--571, 2010.

\bibitem{Moser1971}
J\"urgen Moser.
\newblock A sharp form of an inequality by N. Trudinger.
\newblock \emph{Indiana University Mathematics Journal}, 20:1077--1092, 1971.

\bibitem{Pazy1983}
Amnon Pazy.
\newblock \emph{Semigroups of Linear Operators and Applications to Partial Differential Equations}, volume 44 of Applied Mathematical Sciences.
\newblock Springer, New York, 1983.

\bibitem{Ralston1982}
James~V. Ralston.
\newblock Gaussian beams and the propagation of singularities.
\newblock In \emph{Studies in Partial Differential Equations}, volume 23 of MAA Studies in Mathematics, pages 206--248. Mathematical Association of America, Washington, DC, 1982.

\bibitem{Rauch2012}
Jeffrey Rauch.
\newblock \emph{Hyperbolic Partial Differential Equations and Geometric Optics}, volume 133 of Graduate Studies in Mathematics.
\newblock American Mathematical Society, Providence, RI, 2012.

\bibitem{Taha2005}
Hassan Taha.
\newblock Semi classical measures and Maxwell's system.
\newblock Preprint, arXiv:math/0501127, 2005.

\bibitem{Yu1DCompanion2026}
Cheng Yu.
\newblock Weak-to-strong Maxwell--Ohm compactness and finite-energy solutions for a one-dimensional Navier--Stokes--Maxwell system.
\newblock Manuscript, 2026.

\end{thebibliography}
\end{document}